\documentclass[letterpaper,11pt]{amsart}

\usepackage[margin=1.2in]{geometry}
\usepackage{amsmath,amsthm,amssymb, mathtools}
\usepackage{xspace,xcolor}
\usepackage[breaklinks,colorlinks,citecolor=teal,linkcolor=teal,urlcolor=teal,pagebackref,hyperindex]{hyperref}
\usepackage[alphabetic]{amsrefs}
\usepackage{comment}
\usepackage{enumerate}
\usepackage[all]{xy}
\usepackage{tikz-cd}
\usepackage{color}
\usepackage{comment}

\theoremstyle{plain}
\newtheorem{theo}{Theorem}[section]

\newtheorem{lemm}[theo]{Lemma}
\newtheorem{prop}[theo]{Proposition}
\newtheorem{coro}[theo]{Corollary}

\theoremstyle{definition}
\newtheorem{defi}[theo]{Definition}

\newtheorem{exam}[theo]{Example}

\theoremstyle{remark}
\newtheorem{rema}[theo]{Remark}

\newcommand{\bfR}{\mathbf{R}}

\newcommand{\bfT}{\mathbf{T}}

\newcommand{\HH}{\mathbb{H}}
\newcommand{\PP}{\mathbb{P}}
\newcommand{\QQ}{\mathbb{Q}}

\newcommand{\RR}{\mathbb{R}}
\newcommand{\CC}{\mathbb{C}}

\newcommand{\DD}{\mathbb{D}}

\newcommand{\ZZ}{\mathbb{Z}}

\newcommand{\cD}{\mathcal{D}}

\newcommand{\cH}{\mathcal{H}}

\newcommand{\cK}{\mathcal{K}}

\newcommand{\cM}{\mathcal{M}}

\newcommand{\cO}{\mathcal{O}}
\newcommand{\cP}{\mathcal{P}}
\newcommand{\cQ}{\mathcal{Q}}

\newcommand{\cS}{\mathcal{S}}

\newcommand{\lind}[1]{_{(#1)}}
\newcommand{\ubind}[1]{^{[#1]}}

\newcommand{\lsta}{_{*}}

\newcommand{\sta}{^{*}}

\newcommand{\dual}{^{\vee}}

\newcommand{\lround}[1]{\left\lfloor #1 \right\rfloor}

\DeclareMathOperator{\Hom}{Hom}

\DeclareMathOperator{\rank}{rank}

\DeclareMathOperator{\codim}{codim}

\DeclareMathOperator{\depth}{depth}

\DeclareMathOperator{\im}{im}

\DeclareMathOperator{\gr}{gr}
\DeclareMathOperator{\an}{an}

\DeclareMathOperator{\Spec}{Spec}

\newcommand{\SheafHom}{\mathcal{H}om}

\newcommand{\SheafExt}{\mathcal{E}xt}

\DeclareMathOperator{\coh}{coh}

\newcommand{\duBois}{\underline{\Omega}}

\DeclareMathOperator{\DR}{DR}

\DeclareMathOperator{\IC}{IC}
\DeclareMathOperator{\IH}{IH}

\DeclareMathOperator{\MHM}{MHM}
\DeclareMathOperator{\HM}{HM}

\DeclareMathOperator{\dR}{dR}

\DeclareMathOperator{\lcd}{lcd}
\DeclareMathOperator{\lcdef}{lcdef}
\DeclareMathOperator{\Ish}{Ish}

\usepackage{soul}

\begin{document}
	\thanks{The authors were partially supported by NSF grant DMS-2301463.}
	
	\subjclass[2020]{14B05, 14B15, 14F10, 14M25, 32S35, 52B20}
	
	\author{Hyunsuk Kim}
	
	\address{Department of Mathematics, University of Michigan, 530 Church Street, Ann Arbor, MI 48109, USA}
	
	\email{kimhysuk@umich.edu}	
	
	\author{Sridhar Venkatesh}
	
	\address{Department of Mathematics, University of Michigan, 530 Church Street, Ann Arbor, MI 48109, USA}
	
	\email{srivenk@umich.edu}
	
	\begin{abstract}
        Given an affine toric variety $X$ embedded in a smooth variety, we prove a general result about the mixed Hodge module structure on the local cohomology sheaves of $X$. As a consequence, we prove that the singular cohomology of a proper toric variety is mixed of Hodge--Tate type. Additionally, using these Hodge module techniques, we derive a purely combinatorial result on rational polyhedral cones that has consequences regarding the depth of reflexive differentials on a toric variety.
        
        We then study in detail two important subclasses of toric varieties: those corresponding to cones over simplicial polytopes and those corresponding to cones over simple polytopes. Here, we give a comprehensive description of the local cohomology in terms of the combinatorics of the associated cones, and calculate the Betti numbers (or more precisely, the Hodge--Du Bois diamond) of a projective toric variety associated to a simple polytope. 
	\end{abstract} 
	
	\title[Local cohomology and singular cohomology of toric varieties]{Local cohomology and singular cohomology of toric varieties via mixed Hodge modules}
	
	\maketitle
	
	
	
	\section{Introduction}

In this article, we prove various results about the local cohomology and singular cohomology of toric varieties. The main idea, inspired by \cite{Mustata-Popa22:Hodge-filtration-local-cohomology}, is to exploit the fact that the local cohomology sheaves have a refined structure of \textit{mixed Hodge modules}. 

\noindent \textbf{Local cohomology.} Given a closed subvariety $X$ of a smooth, irreducible complex algebraic variety $Y$, we consider the local cohomology sheaves $\cH_{X}^{q}(\cO_Y)$, where $q$ is a positive integer.
    
    The local cohomology of $\cO_{Y}$ starts from the codimension of $X$ in $Y$, i.e., we have
	$$ \codim_{Y}(X) = \min\{ q : \cH_{X}^{q}(\cO_{Y}) \neq 0 \}.$$
	The \textit{local cohomological dimension}, defined as the other end,
	$$ \lcd(Y, X)  := \max \{ q : \cH_{X}^{q}(\cO_{Y}) \neq 0 \}\footnote{We note that the local cohomological dimension is sometimes defined as the minimal integer $q$ such that the local cohomology of \textit{all} quasi-coherent sheaves $\cM$ on $Y$ vanishes in cohomological degrees $> q$, instead of just for $\cO_{Y}$. However, the two definitions agree (see \cite{Ogus-lcd}*{Proposition 2.1}).},$$
    is a much more subtle and well studied invariant.
    While $\lcd(Y,X)$ depends on the choice of the embedding $X \hookrightarrow Y$, the quantity 
    $$\lcdef(X) := \lcd(Y, X)- \codim_{Y}(X)$$
    only depends on $X$, and we call this the \textit{local cohomological defect}, following the terminology in \cite{Popa-Shen:DuBoisLCDEF}. From the description of the local cohomology in terms of the \v{C}ech complex, it is clear that $\lcd(Y, X)$ cannot exceed the number of equations locally defining $X$. In particular, if $X$ is a local complete intersection (lci) variety, then $\lcd(Y, X) = \codim_{Y}(X)$, i.e., $\lcdef(X) = 0$. Therefore, $\lcdef(X)$ can be thought of as a coarse measure of how far the variety $X$ is from being lci. For a Cohen--Macaulay variety $X$ of dimension $n$, we have $\lcdef(X) \leq \max\{0,n-3\}$ by \cite{Dao-Takagi}.

\noindent \textbf{Mixed Hodge modules.} Given $X$ an irreducible smooth complex projective variety, with $X^{\an}$ denoting its underlying complex manifold, each singular cohomology group $H^k(X^{\an},\mathbb{Q})$ is endowed with a \textit{pure Hodge structure of weight $k$}, i.e. there is a decomposition
\[ H^k(X^{\an},\mathbb{C}) \simeq \bigoplus_{p+q=k} H^{p,q}(X), \]
such that the complex conjugation action on $H^k(X^{\an},\mathbb{C}) = H^k(X^{\an},\mathbb{Q})\otimes_{\mathbb{Q}} \mathbb{C}$ gives 
$$\overline{H^{p,q}(X)} \simeq H^{q,p}(X).$$
This data is equivalent to the data of a decreasing filtration $F^\bullet$ (called the \textit{Hodge filtration}) on $H^k(X^{\an},\mathbb{C})$ by complex subspaces, satisfying the condition
\[ F^p \oplus \overline{F^{k+1-p}} = H^k(X^{\an},\mathbb{C}). \]

Deligne \cite{Deligne-Hodge-2, Deligne-Hodge-3} later generalized this to any complex algebraic variety $X$, and showed that each singular cohomology group $H^k(X^{\an},\mathbb{Q})$ is endowed with a \textit{mixed Hodge structure} i.e., along with the Hodge filtration, there is also a filtration $W_\bullet$ on $H^k(X^{\an},\mathbb{Q})$ (called the \textit{weight filtration}) satisfying the condition that each graded piece $\gr^W_l H^k(X^{\an},\mathbb{Q})$ with the induced Hodge filtration is a pure Hodge structure of weight $l$.

The theory of mixed Hodge modules developed by Morihiko Saito \cite{saito1988modulesdeHodge, saito1990mixedHodgemodules} is a vast generalization of the theory of mixed Hodge structures. The basic object is a \textit{pure Hodge module} of weight $k$, where $k$ is an integer. Roughly speaking, the data of a pure Hodge module on a smooth variety $X$ consists of a (regular, holonomic, left) $\cD_X$-module, where $\cD_X$ denotes the sheaf of differential operators on $X$, along with some additional pieces of data satisfying numerous conditions. Another important notion is that of a \textit{mixed Hodge module} $\cM$. The main additional data here is a filtration called the \textit{weight filtration} $W_\bullet$ such that each graded piece $\gr^W_l \cM$ is a pure Hodge module of weight $l$. Saito describes the category of mixed Hodge modules $\MHM(X)$ and its bounded derived category $D^b\MHM(X)$, and shows that it has the six functor formalism. For a singular variety $X$, a mixed Hodge module is defined by embedding it to a smooth variety and by considering the objects supported on $X$. It turns out that the category of mixed Hodge modules does not depend on the choice of the embedding.\footnote{For varieties that are not embeddable to a smooth variety, one can still consider Hodge modules by locally embedding it into smooth varieties and requiring suitable compatibility conditions. However, all the varieties we consider are quasi-projective, hence embeddable into a smooth variety.} The category $\MHM(pt)$ is just the category of mixed Hodge structures.

By the structure theorem of Hodge modules \cite{saito1988modulesdeHodge}, any polarizable variation of Hodge structures on an open subset of $X_{\mathrm{reg}}$ of weight $w - \dim X$ extends to a pure Hodge module on $X$ of weight $w$ in a unique and controlled manner, where $X_{\mathrm{reg}}$ denotes the smooth locus of $X$. In particular, if we consider the trivial variation of Hodge structures $\QQ_{X_{\mathrm{reg}}}$ on $X_{\mathrm{reg}}$, the associated pure Hodge module on $X$ is called the \textit{intersection cohomology Hodge module} $\IC_{X}^{H}$, which is of weight $\dim X$.

\noindent \textbf{The trivial Hodge module and the Du Bois complex.} The main focus of this article is the so called \textit{trivial Hodge module} $\QQ^{H'}_X := \QQ^H_X[n] \in D^b\MHM(X)$ and its dual $\DD\QQ^{H'}_X$. Both these objects are natural complexes of mixed Hodge modules associated to $X$ and are intimately related to the singular cohomology and the local cohomology of $X$. For instance, pushing forward $\QQ^{H'}_X$ to a point computes the mixed Hodge structure on the singular cohomology of $X$ (see \S\ref{section:hodge-modules}) whereas the dual $\DD\QQ^{H'}_X$ is closely related to the local cohomology of $X$ as follows. If $X \hookrightarrow Y$ is a closed embedding to a smooth variety, then the local cohomology sheaf $\cH^{q}_{X}(\cO_{Y})$ is the underlying $\cD$-module of the mixed Hodge module $i\lsta \cH^{q -c} (\DD \QQ_{X}^{H'})$, where $c = \codim_{Y}(X)$. Consequently, $\lcdef(X)$ is just the cohomological amplitude of $\DD\QQ^{H'}_X$ (or equivalently, of $\QQ^{H'}_X$). See \S \ref{section:hodge-modules} for more details.

Closely related to these Hodge modules is the {\it $p$-th Du Bois complex} $\duBois_{X}^{p} \in D^b_{\coh}(X)$ and its Grothendieck dual $\bfR\SheafHom_{\cO_{X}}(\duBois_{X}^{p}, \omega_{X}^\bullet)$, where $p$ is a non-negative integer. The complex $\duBois_{X}^{p}$ can be thought of as a better-behaved substitute for the K\"ahler differentials $\Omega_{X}^p$ when $X$ is singular. For instance, the Hodge-to-de Rham spectral sequence degenerates at the $E_1$-page for singular projective varieties if one uses the Du Bois complex in place of the K\"ahler differentials. See \S \ref{section:DB-complex} for more details.

Studying the $p$-th Du Bois complex and its Grothendieck dual leads us to interesting insights about the object $\QQ^{H'}_X$ and its dual $\DD\QQ^{H'}_X$, and hence about the singular cohomology and the local cohomology of $X$. For example, when $X$ is proper, the cohomology groups $\HH^q(X,\duBois^p_X)$ are the graded pieces of the Hodge filtration on the singular cohomology $H^{p+q}(X^{\an},\QQ)$, and so, the Betti numbers can be written as a sum of the \textit{Hodge--Du Bois numbers} $\underline{h}^{p, q}(X) := \dim_{\CC} \HH^q(X,\duBois^p_X)$,
\[ \dim_{\CC}H^{k}(X^{\an},\CC) = \sum_{p+q = k} \underline{h}^{p, q}(X).  \]
On the other hand, \cite{Mustata-Popa22:Hodge-filtration-local-cohomology}*{Corollary 5.3} essentially states that the local cohomological defect of $X$ can be controlled via the vanishing of the sheaves $\SheafExt_{\cO_{X}}^i(\duBois_{X}^{p}, \omega_{X}^{\bullet} )$, or more precisely via the \textit{depth} of $\duBois_X^p$, which is given by
$$ \depth \duBois_{X}^{p} = \min \{ i : \SheafExt_{\cO_{X}}^{-i}(\duBois_{X}^{p}, \omega_{X}^{\bullet} ) \neq 0 \}.$$
These ideas have been explored to prove various results regarding the Hodge module structure on the local cohomology of Schubert varieties in \cite{Perlman-local-coh-schubert} and the local cohomology of secant varieties in \cite{OR-local-coh-secant}.

    
\noindent \textbf{Toric varieties.} A toric variety $X$ is a normal algebraic variety containing a torus $T$ as an open dense subset such that the action of $T$ on itself extends to an action of $T$ on the whole variety $X$. Toric varieties provide an interesting interplay between algebraic geometry and convex geometry since they admit an alternate description in terms of convex geometric objects. To be precise, every $n$-dimensional affine toric variety $X$ is associated to a strongly convex rational polyhedral cone $\sigma \subset N \otimes \mathbb{R}$, where $N$ is a free abelian group of rank $n$. More generally, we have a correspondence between toric varieties and \textit{fans}. See \S \ref{section:toric-var-basic} for details.

The singularities of a toric variety are rather mild, for instance, they have rational singularities and are hence Cohen--Macaulay (in particular, the dualizing complex $\omega^\bullet_X$ is equal to the shifted dualizing sheaf $\omega_X[n]$). However, they are typically far from being lci, which makes their local cohomological defect, or more generally, their local cohomology sheaves highly interesting. One subclass of toric varieties for which the trivial Hodge module $\QQ^{H'}_X$ and its dual are well understood is that of \textit{simplicial toric varieties}, they are precisely those toric varieties which have quotient singularities. In this case, we have $\QQ^{H'}_X \simeq \IC^H_X$ and $\lcdef(X) = 0$.


The Du Bois complexes and their Grothendieck duals admit rather nice descriptions when $X$ is a toric variety. By \cite{Guillen-Navarro-Gainza:Hyperresolutions-cubiques}*{V.4}, the Du Bois complex $\duBois_{X}^{p}$ coincides with the sheaf of reflexive differentials $\Omega_{X}\ubind{p}$. On the other hand, it follows by \cite{Ishida2} that its (shifted) Grothendieck dual is given by:
$$ \bfR\SheafHom_{\cO_{X}}(\duBois_{X}^{p}, \omega_{X}) \simeq \Ish_{X}^{n-p},$$
where $\Ish_{X}^{n-p}$ is the $(n-p)$-th \textit{Ishida complex} of $X$, and $n$ is the dimension of $X$. This is a very explicit complex lying in cohomological degrees $0$ to $n-p$, whose terms consist of structure sheaves of various torus-invariant closed subsets. Moreover, if we take into account the torus action, the Ishida complex decomposes into complexes of finite dimensional vector spaces, each of which can be described purely in terms of the convex geometric objects associated to $X$. 
We refer to \S \ref{section:Ishida-complex} for details, where we also present some results of \cite{Ishida,Ishida2} from scratch, using different methods.


    We now state the results in this paper.
	
	\subsection{Hodge module structure of $\QQ_{X}^{H'}$} For a toric variety $X$, we describe the mixed Hodge module structure on all the cohomologies of the object $\QQ_{X}^{H'}$. Roughly speaking, we prove that $\QQ^{H'}_X$ is built out of intersection cohomology Hodge modules along torus-invariant closed subvarieties of $X$, and their Tate twists (which is an operation that changes the weight of a Hodge module, see \S\ref{section:hodge-modules}). We say that a mixed Hodge module $\cM$ has weights in $[a, b]$ if $\gr_{k}^{W} \cM = 0$ for $k \notin [a, b]$.
    
	\begin{theo} \label{theo:MHM-structure}
		Let $X$ be an $n$-dimensional toric variety corresponding to a fan $\cP$. Let $\cP_{d}$ be the collection of $d$-dimensional faces. First, $\QQ_{X}^{H'}$ has no cohomologies outside degrees $[-(n-3), 0]$.
		\begin{enumerate}
			\item The Hodge module $\cH^{0}\QQ_{X}^{H'}$ has weights in $[2, n]$, and the Hodge module $\cH^{-l} \QQ_{X}^{H'}$ has weights in $[2, n-l-1]$ for $l > 0$.
			\item $\gr_{n}^{W} \cH^{0} \QQ_{X}^{H'} \simeq \IC_{X}^{H}$.
			\item For $k \geq l+1$, the weight graded pieces of the cohomologies of $\QQ_{X}^{H'}$ admit a decomposition formula as follows:
			$$\gr_{n-k}^{W} \cH^{-l} \QQ_{X}^{H'} \simeq \bigoplus_{j \geq 1} \bigoplus_{\lambda \in \cP_{k+ 2j}} \IC_{S_{\lambda}}^{H}(-j)^{a_{\lambda}^{l, j}}, $$
            where $S_{\lambda}$ is the torus-invariant closed subvariety corresponding to the cone $\lambda$ (see \S\ref{section:toric-var-basic}).
			For other pairs $(k, l)$, this module is zero.
			\item The numbers $a_{\lambda}^{l, j}$ only depend on the cone $\lambda$.
		\end{enumerate}
	\end{theo}

    From the mixed Hodge module structure of $\QQ_{X}^{H'}$, it is easy to describe the structure of its dual $\DD \QQ_{X}^{H'}$ which is directly related to local cohomology. Additionally, we show in Appendix \ref{sec:appendix-explicit-calculation} that all the numbers $a_{\lambda}^{l, j}$ can be computed explicitly up to dimension 5 in terms of the Ishida complex and address the limitation of our method in dimension 6. It would be very interesting to know if the numbers $a_{\lambda}^{l, j}$ could be computed explicitly in dimensions $\geq 6$ as well, or better, if a computer program can be written that can take in as input a cone $\lambda$, and give as output these numbers $a_{\lambda}^{l, j}$ associated to $\lambda$.

    We say that a pure Hodge structure $V = \oplus V^{p,q} $ is of \textit{Hodge--Tate type} if the $(p, q)$-part $V^{p,q}$ is zero unless $p = q$. We say that a mixed Hodge structure $V$ is \textit{mixed of Hodge--Tate type} if all the weight graded pieces $\gr_{w}^{W}V$ are pure Hodge structures of Hodge--Tate type. It is a classical fact that the Hodge structures of smooth proper toric varieties are of Hodge--Tate type. Using Theorem \ref{theo:MHM-structure}, we prove the following generalization, which is likely well known to experts.

    \begin{coro}\label{coro:sing-coho-HT-type}
        Let $X$ be a proper toric variety. Then the Hodge structure on $H^{k}(X, \QQ)$ is mixed of Hodge--Tate type.
    \end{coro}

    For experts, the rough shape of Theorem \ref{theo:MHM-structure} is perhaps somewhat expected, since the action of the torus gives strong restrictions on the behavior of $\QQ_{X}^{H'}$. However, the range of the weights, the dimensions of the faces that can occur, and the possible Tate twists described in Theorem \ref{theo:MHM-structure} play a crucial role in the following result regarding the depth of the Du Bois complexes.
	\begin{theo} \label{theo:depth-of-Du-Bois-general}
		Let $X$ be an $n$-dimensional toric variety. Then
		$$ \SheafExt_{\cO_{X}}^{l}(\duBois_{X}^{k}, \omega_{X}) = 0, \quad \text{for } l+k > n,$$
		and 
        $$\SheafExt_{\cO_{X}}^{n-k}(\duBois_{X}^{k}, \omega_{X}) = 0 \text{ for }k \leq n/2.\footnote{We have since found an alternate proof of this statement, see \cite{LCDTV2}*{Theorem 1.4, Proposition 1.5}.}$$
        In particular, we have $\depth(\duBois^k_X) \geq k$ for all $k$, and the inequality is strict for all $0<k \leq n/2$.
	\end{theo}

	The first vanishing is immediate from the fact that $\Ish_{X}^{n-k}$ lies in cohomological degrees $0$ to $n-k$. But the second assertion, which says that the last map of $\Ish_{X}^{n-k}$ is surjective for $k$ up to half the dimension, is much more delicate. As mentioned previously, the Ishida complex can be decomposed into pieces which can be described purely combinatorially without making any references to algebraic geometry. However, we are not able to find a direct combinatorial proof for the surjectivity, except for the case $k = 0,1$. Our proof in fact relies on the Hodge module structure of $\QQ_{X}^{H'}$ in Theorem \ref{theo:MHM-structure} in a crucial way. In spirit, this can be thought of as an analogue of Stanley's work \cite{Stanley:Intersection-cohomology-toric-varieties} which used Hodge theory on toric varieties to derive purely combinatorial consequences regarding polytopes.

    Before we end this subsection, we would like to point out the relation of Theorem \ref{theo:MHM-structure} to the work of Reichelt-Walther \cite{Reichelt-Walther}. Let $X$ be an $n$-dimensional affine toric variety corresponding to a cone $\sigma$, and let $h:(\CC^*)^n \hookrightarrow X$ denote the inclusion of the torus in $X$. Reichelt-Walther compute the Hodge module structure of the mixed Hodge module $h_*\QQ^{H'}_{(\CC^*)^n}$. To be precise, they give a description of the summands of its weight graded pieces in terms of the combinatorics of the cone $\sigma$. It seems possible that this result, along with the exact triangle in \cite{Reichelt-Walther}*{\S 1.4.2}, might lead to an alternate proof of Theorem \ref{theo:MHM-structure} via an inductive procedure. However, an important point to note is that $\QQ^{H'}_X$, unlike $h_*\QQ^{H'}_{(\CC^*)^n}$, is not combinatorially determined (see \cite{LCDTV2}*{Example 1.7}), so it is likely that this strategy might be quite subtle.

    We now discuss two subclasses of affine toric varieties for which Theorem \ref{theo:MHM-structure} can be described in complete detail: the toric varieties corresponding to cones over simplicial polytopes, and cones over simple polytopes. 
    In both these cases, a lot of combinatorial data of the associated cones is reflected in the Hodge module structure of $\QQ^{H'}_X$. While the case of cones over simplicial polytopes is similar to the case of isolated singularities, the results about cones over simple polytopes are much more subtle and rather unexpected.

    \subsection{Cones over simplicial polytopes}

    The first case is that of an affine toric variety associated to a full-dimensional \textit{cone over a simplicial polytope} (see \S\ref{section:toric-var-basic} for the precise definition). The affine toric varieties associated to such cones are precisely those which are simplicial outside the torus fixed point. In other words, their non-simplicial locus is either empty or isolated depending on whether the cone is simplicial or not, respectively. We mention that the behavior of $\QQ_{X}^{H'}$ is similar to the situation of isolated singularities (for example, see \cite{Lopez-Sabbah:Hodge-Lyubeznik} for a related result).

    The following theorem says that these toric varieties achieve the upper bound $\lcdef(X) \leq \max\{0,n-3\}$ of \cite{Dao-Takagi} as long as they are non-simplicial.
	
	\begin{theo} \label{theo:cone-over-simplicial-polytope}
        Let $\sigma$ be a full-dimensional cone of dimension $n \geq 3$ and let $X$ denote the associated affine toric variety. Assume that $\sigma$ is a cone over a simplicial polytope, but is not simplicial i.e. $X$ is simplicial outside the torus fixed point, but is not simplicial. Then $\lcdef(X) = n-3$.
	\end{theo}

    Since $X$ has quotient singularities outside the torus fixed point, the natural morphism $\QQ_{X}^{H'} \to \IC_{X}^{H}$ is an isomorphism away from the torus fixed point. Hence, the summands other than $\IC_{X}^{H}$ appearing in Theorem \ref{theo:MHM-structure} are all supported at the torus fixed point. In this case, it is possible to fully describe these summands.

    \begin{theo} \label{theo:cone-over-simplicial}
    Let $\sigma$ be a full-dimensional cone of dimension $n$ and let $X$ denote the associated affine toric variety. Assume that $\sigma$ is a cone over a simplicial polytope. Let $f_{i}$ be the number of $i$-dimensional faces in $\sigma$. Let
    $$ h_{j} = \sum_{l=j}^{n-1} (-1)^{l-j} {l \choose j} f_{n-1-l}.$$
    Then we have the following description for $a_{\sigma}^{l, j}$:
    $$ a_{\sigma}^{l, j} = \begin{cases}
        h_{j} - h_{j-1} & 1 \leq j < n/2 \text{ and } l = n-2j-1 \\
        0 & \text{otherwise.}
    \end{cases}$$
    \end{theo}
    We remark that these $a_{\sigma}^{n-2j-1,j}$ are precisely the coefficients $g_j$ of the $g$-polynomial (see \cite[\S 2]{Stanley:Intersection-cohomology-toric-varieties}) of a suitable hyperplane section of $\sigma$, which is a simplicial polytope. Thus, by \cite{Braden:Remarks-IC-fans}*{Theorem 1.4}, we have the following consequence: If $a_{\sigma}^{n-2j_0-1,j_0} = 0$ for some $j_0$, then $a_{\sigma}^{n-2j-1,j} = 0$ for all $j \geq j_0$.

    In this case, we are also able to fully describe the cohomologies of the Grothendieck dual of the Du Bois complexes, which are the sheaves $\SheafExt_{\cO_{X}}^{k}(\duBois_{X}^{p}, \omega_{X})$. Note that if $k \neq 0$, these sheaves are supported at the torus fixed point, since $X$ has quotient singularities outside the torus fixed point (see \cite{Mustata-Popa22:Hodge-filtration-local-cohomology}*{Corollary 4.29}). 

    \begin{prop} \label{prop:Gro-dual-isolated-non-simplicial}
        Let $\sigma$ be a full-dimensional cone of dimension $n$ and let $X$ denote the associated affine toric variety. Assume that $\sigma$ is a cone over a simplicial polytope. Let $f_{i}$ be the number of $i$-dimensional faces of $\sigma$, and let $h_{j}$ as in Theorem \ref{theo:cone-over-simplicial}.
        For $l \leq n/2$ and $j > 0$, we have
        $$ \dim_{\mathbb{C}} \SheafExt_{\cO_{X}}^{j}(\duBois_{X}^{n-l}, \omega_{X}) = \begin{cases}
            h_{l} - h_{l-1} & \text{if } j= l \\
            0 & \text{otherwise}.
        \end{cases} $$
        For $l \geq n/2$ and $j > 0$, we have
        $$ \dim_{\mathbb{C}} \SheafExt_{\cO_{X}}^{j}(\duBois_{X}^{n-l}, \omega_{X}) = \begin{cases}
        h_{l-1} - h_{l} & \text{if } j = l-1   \\
        0 & \text{otherwise}.
        \end{cases}
         $$
         In particular, if $n$ is even, then $\SheafExt_{\cO_{X}}^{j}(\duBois_{X}^{n/2}, \omega_{X}) = 0$ for all $j>0$; hence $\duBois_{X}^{n/2}$ has maximal depth.
    \end{prop}

    Just as before, these dimensions are the coefficients $g_j$ (or their negatives) of the $g$-polynomial of a suitable hyperplane section of $\sigma$, which is a simplicial polytope. Thus, again by \cite{Braden:Remarks-IC-fans}*{Theorem 1.4}, we have the following consequence: If $\duBois_{X}^{j_0}$ has maximal depth for some $j_0 < n/2$, then $\duBois_{X}^{j}$ has maximal depth for all $j_0 \leq j \leq n/2$. Similarly, if $\duBois_{X}^{j_0}$ has maximal depth for some $j_0 > n/2$, then $\duBois_{X}^{j}$ has maximal depth for all $j_0 \geq j \geq n/2$.

    \subsection{Cones over simple polytopes}

    The second case is that of an affine toric variety associated to \textit{a cone over a simple polytope} (see \S\ref{section:toric-var-basic} for the precise definition). In fact, the affine toric varieties associated to cones over simple polytopes are precisely those for which all the torus-invariant divisors are simplicial.
    These examples provide a wide class of varieties with local cohomological defect zero. We mention that despite having $\lcdef = 0$, these varieties are very far from being rational homology manifolds (i.e, varieties satisfying $\QQ_{X}^{H'} \simeq \IC_{X}^{H}$).

    \begin{theo} \label{theo:lcdef-of-simple}
        Let $X$ be the affine toric variety associated to a cone over a simple polytope, i.e. $X$ is such that all its torus-invariant divisors are simplicial. Then $\lcdef(X) = 0$.
    \end{theo}

    Theorem \ref{theo:lcdef-of-simple} is a special case of the following general theorem about toric varieties whose codimension $c$ torus-invariant closed subvarieties are all simplicial.
    
	\begin{theo} \label{theo:simple-upperbound-depth}
		Let $X$ be a toric variety such that all of its codimension $c$ torus-invariant closed subvarieties are simplicial. Then 
        $$\SheafExt_{\cO_{X}}^{i} (\duBois_{X}^{l}, \omega_{X}) = 0 \text{ for $i > c$, $\forall l$.} $$
        In particular, we have $\depth(\duBois^l_X) \geq n-c$ for all $l$, and $\lcdef(X) \leq \max\{0,c-1\}$.
	\end{theo}
    The upper bound on $\lcdef(X)$ follows from the $\SheafExt$ vanishing because of Theorem \ref{theo:MP-lcd-intermsof-DuBois} below.

    \begin{rema}
        If $c = 0$ (i.e., if $X$ is a simplicial toric variety), this is immediate from \cite{Mustata-Popa22:Hodge-filtration-local-cohomology}*{Corollary 4.29}. It would be interesting to see if putting a nice restriction on some strata of certain codimension (say, on the depth of its Du Bois complexes) gives a lower bound on the depth of the Du Bois complex of the total space. Our proof of the theorem above relies on some special properties of toric varieties and some delicate combinatorics.
    \end{rema}


    We now give a complete description of $\QQ^{H'}_X$ for toric varieties corresponding to cones over simple polytopes.

    \begin{theo} \label{theo:cone-over-simple}
        Let $\sigma$ be a full-dimensional cone of dimension $n$ which is a cone over a simple polytope, and let $X$ denote the associated affine toric variety. Let $f_{i}$ be the number of $i$-dimensional faces in $\sigma$. Let
        $$ \widetilde{h}_{j} = \sum_{l=0}^{j} f_{n-l} {n-1-l \choose j-l} (-1)^{j-l}.$$
        
        If $j \geq n/2$, we have $a_{\sigma}^{0, j} = 0$. For $j < n/2$, we have the following description of the numbers $a_{\sigma}^{0, j}$:
        $$ a_{\sigma}^{0, j} =\widetilde{h}_{j} - \widetilde{h}_{j-1} = \sum_{l=0}^{j} (-1)^{l} f_{n-l+j} {n - j+l \choose l} = \sum_{l=0}^{j} (-1)^{j-l} f_{n-l} {n-l \choose j - l}.$$
    \end{theo}
    We point out that for all proper faces $\lambda \subset \sigma$, the numbers $a_\lambda^{0,j}$ can be calculated using a similar formula for $\lambda$ because $\lambda$, being a face of a cone over a simple polytope, is still a cone over a simple polytope. 

    \begin{rema}
        The definition of $\widetilde{h}_{l}$ in Theorem \ref{theo:cone-over-simple} is different from $h_{l}$ in Theorem \ref{theo:cone-over-simplicial}. The role of the face numbers are reversed so that the numbers $\widetilde{h}_{l}$ and $h_{l}$ both form a `unimodal symmetric mountain', which is a consequence of hard Lefschetz and Poincar\'{e} duality (see \cite{Stanley:Intersection-cohomology-toric-varieties} for details).
    \end{rema}

    We also give a related result regarding the singular cohomology of projective toric varieties associated to simple polytopes. For such a toric variety, we calculate all its Hodge--Du Bois numbers, and hence its Betti numbers (since they are sums of Hodge--Du Bois numbers). We point out that the Hodge--Du Bois diamond turns out to be highly asymmetric in this case.

    \begin{theo} \label{theo:sing-coho-simple}
        Let $P$ be a simple lattice polytope of dimension $n$ containing the origin in the interior. Consider the projective toric variety $(X_{P}, D_{P})$ associated to $P$ as above (see \S\ref{subsec:sing-coho-simple} for the construction). Let $f_{i}$ be the number of $i$-dimensional faces of $P$. Then we have the following description for the Hodge--Du Bois numbers:
        $$ \underline{h}^{p, q}(X_{P}):= \dim H^{q}(\duBois_{X_{P}}^{p}) = \begin{cases}
            1 & p = q\text{ and } p \neq n-1 \\
            0 & p \neq q \text{ and }q \neq n-1\\
            0 & (p, q) = (n, n-1),(0, n-1) \\
            \displaystyle\sum_{j=0}^{n-p} f_{j-1} (-1)^{j-1} {n-j \choose p} + (-1)^{n-p} & q = n-1, 1\leq p < n-1 \\
            f_{0} - n & (p,q) = (n-1, n-1).
        \end{cases} $$
    \end{theo}

    As an application of Theorem \ref{theo:sing-coho-simple}, we calculate the Hodge-Du Bois numbers of the binomial hypersurface $S = \{x_{0}\cdots x_{n} = y_{0}\cdots y_{n}\} \subset \PP^{2n+1}$ in Example \ref{exam:sing_coh_binomial_hypersurface}.

    \noindent
    \textbf{Outline of the Paper.} We discuss the necessary background on Hodge modules and local cohomology in \S\ref{sec:Preliminaries} and toric varieties in \S\ref{section:toric-varieties}. We then discuss the Ishida complex in detail in \S\ref{section:Ishida-complex}, where we also present some results of \cite{Ishida,Ishida2} from scratch, using different methods. We then prove Theorems \ref{theo:MHM-structure}, \ref{theo:depth-of-Du-Bois-general}, and Corollary \ref{coro:sing-coho-HT-type} in \S\ref{section:Hodge-module-structure-Q}. In \S\ref{section:cone-over-simplicial-polytope} and \S\ref{section:cone-over-simple-polytope}, we prove the results about the toric varieties associated to cones over simplicial polytopes and cones over simple polytopes, respectively. Finally, in Appendix \ref{sec:appendix-explicit-calculation}, we explicitly calculate the numbers appearing in Theorem \ref{theo:MHM-structure} up to dimension 5.
	
	\section{Preliminaries}\label{sec:Preliminaries}
    \subsection{The Du Bois complex}\label{section:DB-complex}
	In \cite{DuBois:complexe-de-deRham}, Du Bois introduced a filtered complex $\duBois_{X}^{\bullet}$ which can be thought of as a replacement of the de Rham complex $\Omega_{X}^{\bullet}$ when $X$ is singular. By taking the graded quotients, the {\it $p$-th Du Bois complex} is defined as
	$$ \duBois_{X}^{p} := \gr_{F}^{p} \duBois_{X}^{\bullet} [p] \in D^b_{\rm coh}(X).$$
    We have a natural comparison map $\Omega_{X}^{\bullet} \to \duBois_{X}^{\bullet}$ of filtered complexes which is an isomorphism if $X$ is smooth, where the filtration on $\Omega_{X}^{\bullet}$ is given by truncation. 

    The Du Bois complex is indeed the `correct' object to consider when $X$ is singular. For example, we have the isomorphism $\CC_{X} \simeq \duBois_{X}^{\bullet}$ and so, the hypercohomology of $\duBois_{X}^{\bullet}$ computes the singular cohomology of $X$. If $X$ is a proper variety, then the spectral sequence computing the singular cohomology degenerates at $E_{1}$. In particular, the filtration given by the spectral sequence agrees with Deligne's Hodge filtration on the singular cohomology of algebraic varieties in the following sense:
    $$ F^{p}H^{k}(X, \CC) = \im (\HH^{k}(X, F^{\geq p}\duBois_{X}^{\bullet}) \to H^{k}(X, \CC)).$$
    Hence, the graded quotient can be expressed as $\gr_{F}^{p} H^{k}(X,\CC) \simeq \HH^{k-p}(X, \duBois_{X}^{p})$.

    By \cite{Guillen-Navarro-Gainza:Hyperresolutions-cubiques}*{V.4}, the Du Bois complex $\duBois_{X}^{p}$ coincides with the sheaf of reflexive differentials $\Omega_{X}\ubind{p}$ when $X$ is a toric variety. In particular, $\duBois_{X}^{p}$ is a sheaf in this case. Since all of the varieties that we deal with are toric, we will not distinguish the Du Bois complex and the sheaf of reflexive differentials throughout this article.

    \subsection{Perverse sheaves}
	For an algebraic variety $X$, the bounded derived category of sheaves of $\QQ$-vector spaces on $X$ contains the full subcategory $D^b_c(X)$, called the \textit{constructible bounded derived category}, whose objects are bounded complexes with constructible cohomology sheaves. Inside $D^b_c(X)$, we have the (abelian) subcategory of \textit{perverse sheaves} which is the heart of a special $t$-structure (called the \textit{perverse} $t$-structure) on $D^b_c(X)$ or equivalently, as the subcategory of constructible complexes $K^\bullet$ which satisfy certain support conditions on the cohomologies of $K^\bullet$ and its Verdier dual $\DD K^{\bullet}$. Under the Riemann--Hilbert correspondence, they are precisely those constructible complexes that correspond to regular holonomic $\cD$-modules. Typical examples of perverse sheaves are intersection complexes associated to local systems on locally closed subsets of $X$; in fact, every perverse sheaf on $X$ is built out of such intersection complexes. For more details on perverse sheaves, we refer to \cite{dCM-Decomposition} and \cite{HTT-Dmodulesbook}.
    
    \subsection{Hodge modules} \label{section:hodge-modules} We briefly recall the results for Hodge modules we will use. See \cite{saito1988modulesdeHodge, saito1990mixedHodgemodules} for basic facts and results. For an algebraic variety $X$, we denote by $\HM(X, w)$ and $\MHM(X)$ the category of pure Hodge modules of weight $w$ and mixed Hodge modules on $X$, respectively. On a smooth variety $W$, mixed Hodge modules are regular holonomic $\cD$-modules $\cM$ with a filtration $F_{\bullet}$ by coherent sheaves (called the Hodge filtration) and a filtration $W_{\bullet}$ by $\cD$-modules, with the data of an isomorphism $\alpha \colon \DR \cM \xrightarrow{\simeq} K \otimes_{\QQ} \CC$, where $K$ is a $\QQ$-perverse sheaf, satisfying certain complicated conditions. For singular varieties $X$ embeddable to a smooth variety $W$, the mixed Hodge modules on $X$ are simply the ones on $W$ supported on $X$. The category $\MHM(X)$ of mixed Hodge modules on $X$ does not depend on the embedding. The objects in the category $\HM(X, w)$ are the objects $\cM$ in $\MHM(X)$ such that $\gr_{w'}^{W}\cM = 0$ for all $w' \neq w$. The most important thing throughout this article is that they have the six functor formalism which is compatible with the corresponding functors for the underlying perverse sheaves.

    We first describe the structure of pure Hodge modules on $X$. For a closed irreducible subvariety $Z$, we say that a pure Hodge module $\cM$ has strict support $Z$ if it is supported on $Z$ and has no nonzero subobjects or quotients supported on a strictly smaller subset of $Z$. Any pure Hodge module $\cM$ on $X$ admits a decomposition by strict support, that is, there is a decomposition
    $$ \cM = \bigoplus_{Z \subset X} \cM_{Z}$$
    such that $\cM_{Z}$ has strict support $Z$, where the sum runs over all irreducible subvarieties of $X$. Moreover, any pure Hodge module $\cM_{Z}$ of weight $w$ with strict support $Z$ is a polarizable variation of Hodge structures of weight $w - \dim Z$ on an open subset of $Z_{\mathrm{reg}}$. Conversely, any variation of Hodge structures on an open subset of the smooth locus of $Z$ can be uniquely extended to a pure Hodge module on $X$ with strict support $Z$, where the underlying perverse sheaf is the intersection complex associated to the variation of Hodge structures. For example, if $X$ is an algebraic variety of dimension $n$, then there is a trivial variation of Hodge structures on the regular locus $X_{\mathrm{reg}}$ which defines the \textit{intersection cohomology Hodge module} $\IC_{X}^{H}$. This is a pure Hodge module of weight $n$ on $X$.
    
    The graded de Rham complex can be associated to an object in the derived category of mixed Hodge modules. If $X$ is a smooth variety and $\cM$ is a mixed Hodge module on $X$ (by considering its underlying filtered left $\cD$-module), then for $k \in \ZZ$, the graded de Rham complex is given by
	$$ \gr_{k}^{F} \DR \cM = [ \gr_{k}^{F} \cM \to \Omega_{X}^{1} \otimes \gr_{k+1}^{F} \cM \to \ldots \to \Omega_{X}^{\dim X} \otimes \gr_{k+ \dim X}^{F}\cM]$$
	which sits in cohomological degrees $-\dim X, \ldots, 0$. Since the morphisms between Hodge modules are strict with respect to the Hodge filtration, we have an exact functor
	$$\gr_{k}^{F}\DR: D^{b}(\MHM(X)) \to D^{b}_{\coh}(X).$$
	Sometimes, we will omit the Hodge filtration $F$, and write simply $\gr_{k} \DR$. If $X$ is singular, but embeddable in a smooth variety $W$, we can consider objects in $\MHM(X)$ as objects in $\MHM(W)$ supported on $X$ and consider the graded de Rham complex on $W$. It is a fact that the graded de Rham complex is indeed an object in $D_{\coh}^{b}(X)$ and does not depend on the embedding in the ambient smooth variety.
	
	We describe the behavior of Hodge modules under the duality functor $\DD$. For $\cM \in D^{b}(\MHM(X))$, the graded de Rham complex behaves in the following way (see \cite{saito1988modulesdeHodge}*{\S2.4}):
	$$ \gr_{k} \DR \DD (\cM) \simeq \bfR \SheafHom_{\cO_{X}} (\gr_{-k} \DR \cM, \omega_{X}^{\bullet}). $$
	For $\cM \in D^{b}(\MHM(X))$, we have $\cH^{l} \DD \cM\simeq \DD\cH^{-l} \cM$, and for a single mixed Hodge module $\cM \in \MHM(X)$, we have $\DD\gr_{w}^{W} \cM \simeq \gr_{-w}^{W} \DD \cM$ (see \cite{saito1990mixedHodgemodules}*{\S2.a}). These statements are stated in terms of a single object $\cM \in \MHM(X)$, but they automatically generalize to an object in the derived category since $\DD$ is contravariant exact.
	
	Next, we describe the behavior of a Hodge module under smooth pull-back (see \cite{saito1990mixedHodgemodules}*{\S2.d}). Let $\pi : X \to Y$ be a smooth morphism of relative dimension $r$. Then for $\cM \in D^{b}(\MHM(Y))$, we have
	$$ \pi\sta \cH^{l} \cM \simeq \left(\cH^{l+r} \pi\sta \cM\right) [-r],$$
	and
	$$ \pi\sta \gr_{k}^{W} \cH^{l} \cM \simeq \left(\gr_{k+r}^{W} \cH^{l+r}\pi\sta \cM \right)[-r].$$

    For a Hodge module $\cM$ and an integer $r$, we have the \textit{Tate twist} $\cM(r)$ which shifts the Hodge filtration by $F_{\bullet} \cM(r) = F_{\bullet - r}\cM$ and its underlying perverse sheaf $K$ by $K(r) = K \otimes_{\QQ} (2\pi i)^r \QQ$. This decreases the weight of $\cM$ by $2r$.

    Lastly, we state the Decomposition theorem for Hodge modules. For a proper morphism $\pi \colon Y \to X$ and a pure Hodge module $\cM$ of weight $w$ on $Y$, we have a (non-canonical) decomposition
    $$ \pi\lsta \cM \simeq \bigoplus_{j} \cH^{j} \pi\lsta \cM [-j],$$
    where $\cH^{j}\pi\lsta \cM$ is a pure Hodge module of weight $w + j$ on $X$.

    \subsection{Mixed Hodge structure on singular cohomology}
    While Deligne's theory puts a mixed Hodge structure on the singular cohomology of an algebraic variety, Saito's theory also provides a way to put a mixed Hodge structure on it, using the trivial Hodge module. The category of mixed Hodge modules over a point can be identified with the category of mixed Hodge structures. Hence, we have the Hodge module $\QQ_{\mathrm{pt}}^{H}$ with weight zero given by the following mixed Hodge structure $(V, W_{\bullet}, F^{\bullet})$:
	$$V = \QQ, \quad W_{-1}V = 0, \quad W_{0}V = V,\quad  F^{0}V_{\CC} = V_{\CC}, \quad F^{1}V_{\CC} = 0.$$
    For an arbitrary variety $X$, one defines
    $$ \QQ_{X}^{H} = a_{X}\sta \QQ_{\mathrm{pt}}^{H} \in D^{b}(\MHM(X)),$$
    where $a_{X} \colon X \to \{ \mathrm{pt}\}$ is the structure morphism. This is an object in the derived category of mixed Hodge modules on $X$, whose underlying constructible complex is $\QQ_{X}$. We call $\QQ_{X}^{H'}:=\QQ_{X}^{H}[\dim X]$ the \textit{trivial Hodge module}.

    Since the six functor formalism for Hodge modules is compatible with the underlying perverse sheaves, the $\cH^{k} (a_{X\ast} \QQ_{X}^{H}) \in \MHM(\mathrm{pt})$ is a mixed Hodge structure whose underlying vector space is $H^{k}(X, \QQ)$.

    In \cite{Saito-MixedHodgecomplexes}, Saito shows that the two ways to put a mixed Hodge structure on $H^{k}(X, \QQ)$ agree in a sheaf-theoretic sense. For example, it is a consequence of \cite{Saito-MixedHodgecomplexes}*{Theorem 0.2} that the graded de Rham complex of $\QQ_{X}^{H}$ is related to the Du Bois complex in the following way:
    $$ \gr_{k}\DR_{X} \QQ_{X}^{H}[n] \simeq \duBois_{X}^{-k}[n+k].$$

    In particular, if $X$ is proper, then the graded pieces $\gr_{F}^{p} H^{k}(X)$ in Saito's theory can also be computed by $\HH^{k-p}(X, \duBois_{X}^{p})$, using the commutativity of the graded de Rham functor with proper push-forward \cite{saito1988modulesdeHodge}*{2.3.7}. We remark that one can also get the isomorphism above by using \cite{Mustata-Popa22:Hodge-filtration-local-cohomology}*{Proposition 5.5} and duality, when $X$ is embeddable in a smooth variety.
    
	\subsection{Local cohomology} \label{subsec:local-cohomology}
	Let $X$ be a closed algebraic subvariety of a smooth, irreducible complex algebraic variety $Y$ of dimension $n$. For a quasi-coherent $\cO_{Y}$-module $\cM$, the $q$-th local cohomology sheaf $\cH_{X}^{q}\cM$ is the $q$-th cohomology of the derived functor of $\underline{\Gamma}_{X}$, given by the subsheaf of local sections set theoretically supported on $X$.

	Following \cite{Mustata-Popa22:Hodge-filtration-local-cohomology}\footnote{We remark that the role of $X$ and $Y$ is different from \cite{Mustata-Popa22:Hodge-filtration-local-cohomology}. We chose the notation in this way because the embedding is irrelevant for our purposes, we actually work on the singular variety $X$ itself.}, we have the following characterization of the local cohomological dimension and the local cohomological defect in terms of the vanishing of the cohomologies of the Grothendieck duals of the Du Bois complexes.
	\begin{theo}[\cite{Mustata-Popa22:Hodge-filtration-local-cohomology}*{Corollary 5.3}] \label{theo:MP-lcd-intermsof-DuBois}
		Let $X$ be a subvariety of a smooth variety $Y$. Then for any integer $c$, the following are equivalent:
		\begin{enumerate}
			\item $\lcd(Y, X) \leq c$.
			\item $\SheafExt_{\cO_{Y}}^{j+i+1} (\duBois_{X}^{i}, \omega_{Y}) = 0$ for all $j \geq c$ and $i \geq 0$.
		\end{enumerate}
		Equivalently, using Grothendieck duality for the inclusion $X \hookrightarrow Y$, the following are equivalent:
		\begin{enumerate}
			\item $\lcdef(X) \leq c$.
			\item $\SheafExt_{\cO_{X}}^{j+i+1} (\duBois_{X}^{i}, \omega_{X}^{\bullet}[-\dim X]) = 0$ for all $j \geq c$ and $i \geq 0$.
		\end{enumerate}
	\end{theo}

	The main tool in \cite{Mustata-Popa22:Hodge-filtration-local-cohomology} is the structure of a \textit{mixed Hodge module} on the local cohomology modules $\cH_{X}^{q}(\cO_{Y})$. Denote by $i : X \hookrightarrow Y$ the closed immersion, and by $j$ the open immersion from the complement $U = Y \setminus X$. Let $\dim X= d_{X}$, $\dim Y = d_{Y}$, and $\codim_{Y}X = c = d_{Y} - d_{X}$. Then we have an exact triangle
	$$ i\lsta i^{!} \QQ_{Y}^{H'} \to \QQ_{Y}^{H'} \to j\lsta \QQ_{U}^{H'} \xrightarrow{+1}$$
	in the derived category of Hodge modules on $Y$, and the underlying $\cD$-modules of the cohomologies of the first term in the triangle are exactly the local cohomologies:
	$$ \cH^{q}(i\lsta i^{!} \QQ_{Y}^{H'}) = \cH_{X}^{q}(\cO_{Y}).$$
    
	In this way, the sheaves $\cH_{X}^{q}(\cO_{Y})$ carry a Hodge filtration. Since we have a natural isomorphism $i^{!} \simeq \DD i\sta \DD$, we have
    $$ i^{!} \QQ_{Y}^{H'} \simeq \DD i\sta \DD (\QQ_{Y}^{H}[d_{Y}]) \simeq \DD i\sta \QQ_{Y}^{H}[d_{Y}](d_{Y}) \simeq \DD (\QQ_{X}^{H}[d_{Y}] (d_{Y})) \simeq (\DD \QQ_{X}^{H'})[-c](-d_{Y}).$$

    Therefore, we have $\cH^{q}(i^{!} \QQ_{Y}^{H'}) \simeq \cH^{q-c}(\DD \QQ_{X}^{H'})(-d_{Y})$. This also allows us to compute the local cohomological dimension purely topologically. Since the underlying perverse sheaf of $\cH_{X}^{q}(\cO_{Y})$ is exactly $\prescript{p}{}{\cH}^{q-c}(\DD \QQ_{X}[d_{X}]) \simeq \DD \left(\prescript{p}{}{\cH}^{c-q}\QQ_{X}[d_{X}] \right)$ (ignoring the Tate twists), we have
	$$\lcdef(X) = \max \{ j \in \ZZ : \prescript{p}{}{\cH}^{-j}(\QQ_{X}[d_{X}]) \neq 0\} = \max \{ j \in \ZZ : \cH^{-j} \QQ_{X}^{H'} \neq 0\}.$$
	We refer to \cite{RSW-lcdef-intermsofTopology} for more details.
	
	\section{Toric and convex geometry}\label{section:toric-varieties}
	\subsection{Toric varieties} \label{section:toric-var-basic}
	We follow \cite{Fulton-ToricVar, CoxLittleSchenck-ToricVar} for general notions of toric varieties. To a strongly convex rational polyhedral cone $\sigma$, we associate an affine toric variety $X_{\sigma} = \Spec \CC [\sigma\dual \cap M]$. In general, to a fan $\cP$, we associate a toric variety by gluing the affine toric varieties corresponding to the cones of $\cP$.
	
	Before going into toric varieties, we set up some notation for convex cones. From now on, all cones are strongly convex rational polyhedral. Let $N$ be a free abelian group of rank $n$ and let $M:= \Hom_{\ZZ}(N, \ZZ)$. Denote $N_{\RR} := N \otimes_{\ZZ} \RR$ and $M_{\RR} := M \otimes_{\ZZ} \RR$. Let $\sigma$ be a cone in $N_{\RR}$. We denote $\cP$ the collection of all faces of $\sigma$ and view $(\cP, \subseteq)$ as a graded poset. For an integer $m \in [0, n]$, we denote by
	$$ \cP_{m} = \{ \lambda \in \cP : \dim \lambda = m \}. $$
	For $\mu \in \cP$, we set
	$$ \cP^{\subset \mu} := \{\lambda \in \cP : \lambda \subset \mu\}, \quad \cP^{\supset \mu} := \{ \lambda \in \cP: \lambda \supset \mu\}.$$
	Also, we set $\cP_{m}^{\subset \mu}:= \cP_{m} \cap \cP^{\subset \mu} $ and $\cP_{m}^{\supset \mu}:= \cP_{m} \cap \cP^{\supset \mu}$.
	
	Let $\mu, \tau \in \cP$.
	\begin{enumerate}
		\item $\tau\dual := \{ u \in M_{\RR} : \langle u , v \rangle \geq 0 \text{ for all }v \in \tau \} $
		\item $\tau^{\perp} := \{ u \in M_{\RR} : \langle u , v \rangle = 0 \text{ for all }v \in \tau \} $
		\item $\tau\sta_{\circ} := (\tau^{\perp} \cap \sigma\dual \cap M) \setminus \bigcup_{\tau \subsetneq \nu} (\nu^{\perp} \cap \sigma\dual \cap M)$
		\item $\langle \tau \rangle \subset N_{\RR}$ is the subspace spanned by $\tau$
		\item $d_{\tau} := \dim_{\RR} \langle\tau\rangle$
		\item We say $\sigma$ is \textit{full-dimensional} if $d_{\sigma} = \rank_{\ZZ} N$.
		\item $\tau$ is \textit{simplicial} if the 1-dimensional faces (i.e. rays) of $\tau$ are linearly independent over $\RR$ in $N_{\RR}$.
		\item We say $\sigma$ is \textit{simple in dimension $c$} if for all $c$-dimensional faces $\tau$, the image of $\sigma$ in $N_{\RR}/ \langle \tau \rangle$ (by projection) is simplicial. We refer to \cite{Grunbaum:polytope-book}*{\S4.5} for similar generalized notions for simple polytopes.
        \item We say $\sigma$ is \textit{a cone over a simple polytope} if it is simple in dimension 1. This notation comes from the fact that a hyperplane section of $\sigma$ would be a simple polytope.
        \item We say $\sigma$ is \textit{a cone over a simplicial polytope} if every proper face $\tau \subsetneq \sigma$ is simplicial. This notation comes from the fact that a hyperplane section of $\sigma$ would be a simplicial polytope.
	\end{enumerate}
	
	\begin{rema} \label{rema:grading-and-face-relations}
		We remark that there is an order-reversing one-to-one correspondence between the faces of $\sigma$ and the faces of $\sigma\dual$ by sending $\tau$ to $\tau^{\perp} \cap \sigma^{\dual}$. We also point out that $\tau_{\circ}\sta$ gives a partition of the set $\sigma\dual \cap M$. It is straightforward to check that for $u \in \tau_{\circ}\sta$, we have
		$$ u \in \mu^{\perp} \cap \sigma\dual \cap M \quad \text{if and only if} \quad \mu \subset \tau.$$
	\end{rema}
	
	\noindent
	We briefly describe the structure of affine charts, torus-invariant closed subsets and the orbits, following \cite{Fulton-ToricVar}*{\S 3.1}. Let $X =\Spec \CC[\sigma\dual \cap M]$ be the affine toric variety associated to a cone $\sigma$. For an $r$-dimensional face $\tau$ of $\sigma$, we get an irreducible torus-invariant subvariety $S_{\tau}$ of codimension $r$ given by $ \Spec \CC[\sigma\dual \cap \tau^{\perp} \cap M]$. This is the affine toric variety corresponding to the cone $\overline{\sigma}_{\tau}$, where $\overline{\sigma}_{\tau}$ is the image of $\sigma$ under the projection map $N_{\RR} \to N_{\RR} / \langle \tau \rangle$. The lattice and the dual lattice of $S_{\tau}$ is given by
	$$ N_{\tau} := \frac{N}{N \cap \langle \tau \rangle}, \qquad M_{\tau} := M \cap \tau^{\perp}.$$
	We denote by $O_{\tau} = \Spec \CC[M_{\tau}]$ the torus orbit corresponding to $\tau$, and $U_{\tau} =\Spec \CC[\tau^{\dual} \cap M]$ the affine chart of $X$ corresponding to $\tau$. We have a diagram of torus equivariant morphisms
	$$ \begin{tikzcd}
		U_{\tau} \ar[r, hook] \ar[d]& X_{\sigma} \ar[d] \\ O_{\tau} \ar[r, hook] & S_\tau.
	\end{tikzcd}$$
	Here, the horizontal arrows are open immersions. Also, after fixing a non-canonical splitting $N = N_{\tau} \oplus (N \cap \langle\tau\rangle)$ and the corresponding splitting $M = M_{\tau} \oplus M'$, we can identify the vertical map $U_{\tau} \to O_{\tau}$ as the projection $U_{\tau} = V_{\tau} \times O_{\tau} \to O_{\tau}$, where $V_{\tau}$ is the full-dimensional toric variety $\Spec \CC[\tau\dual \cap M']$, by viewing $\tau$ as a cone in $\langle \tau \rangle \cap N_{\RR}$.
	
	We also mention that by \cite{CoxLittleSchenck-ToricVar}*{Theorem 9.2.5}, toric varieties are normal and they have rational singularities, hence Cohen--Macaulay.

	\subsection{Differential forms on toric varieties} \label{section:diff-forms-on-toric}
	We briefly discuss how differential forms work on toric varieties. Note that $M$ can be identified with the group of characters of the torus $\Hom(\bfT, \CC^{\times})$. Hence, for $u \in M$, one can associate a differential 1-form $d \log \chi^{u} := \chi^{-u} \cdot d\chi^{u}$ on $\bfT$, where $\chi^{u} : \bfT \to \CC^{\times}$ is the character corresponding to $u$. One can easily check that this transformation is additive and induces an isomorphism
	$$ M \otimes_{\ZZ} \cO_{\bfT} \simeq \Omega_{\bfT}^{1}.$$
	Let's observe how the 1-form $d \log \chi^{u}$ behaves near the torus invariant divisors. Pick a ray $\rho$ in the fan and consider the affine chart $U_{\rho} = \Spec \CC[\rho\dual \cap M]$ and the torus orbit $O_{\rho} = \Spec \CC[\rho^{\perp} \cap M]$. If $\langle u , \rho \rangle \equiv 0$, then $\chi^{u}$ is an invertible function on $U_{\rho}$. Hence, $\chi^{-u} \cdot d\chi^{u}$ is a differential 1-form on $U_{\rho}$. If $u \notin \rho^{\perp}$, then we assume that $u$ pairs positively with $\rho$, since $d\log \chi^{-u} = - d\log \chi^{u}$. Since the zero locus of $\chi^{u}$ in $U_{\rho}$ is exactly the orbit $O_{\rho}$, we see that $d\log \chi^{u}$ extends to a logarithmic differential form on the open locus of the torus-invariant divisor. Hence, we can show that these differential forms extend as logarithmic differential forms on the whole space $X$ (in a suitable sense), and give an isomorphism
	$$ M \otimes_{\ZZ} \cO_{X} \simeq \Omega_{X}\ubind{1} (\log D),$$
	where $D$ is the sum of the torus-invariant divisors. Here, $\Omega_{X}\ubind{1} (\log D) := j\lsta \Omega_{X \setminus Z}^{1} (\log D|_{X \setminus Z})$ where $Z$ is the union of codimension 2 torus-invariant subspaces, and $j : X \setminus Z \to X$ is the open inclusion. From this, we can see that
	$$ \bigwedge^{l} M \otimes_{\ZZ} \cO_{X} \simeq \Omega_{X}\ubind{l} (\log D),$$
	where $\Omega_{X}\ubind{l}(\log D)$ is defined analogously.
	
	\subsection{The intersection cohomology Hodge module on toric varieties} \label{sec:ICTV-summary}
	We briefly review the results in \cite{KV:ICTV}. Let $X$ be an $n$-dimensional affine toric variety associated to a full-dimensional cone $\sigma$. Then we consider the stalk of the intersection cohomology complex $\IC_{X}$ at the torus fixed point $x_{\sigma}$, written as a generating function as follows:
	$$ H_{\sigma}(q) = q^{n} \cdot \sum_{j} \dim \cH^{j} \left( i_{x_{\sigma}}\sta \IC_{X} \right) q^{j}. $$
	These polynomials are computable inductively, either following \cite{KV:ICTV} or \cite{Fieseler-ICprojtoric} and they only depend on the combinatorial data of $\sigma$. We point out that if $n >0$, $H_{\sigma}(q)$ is a polynomial of degree less than $n$ and only the even powers of $q$ show up in the terms. Indeed, the first part of the assertion is by \cite{dCM-Decomposition}*{\S2.1.(12)} and the second assertion is implicit in \cite{KV:ICTV}*{Theorem 1.1}. We also consider the degree zero part of the graded de Rham complex, written as a generating function as follows:
	$$ \dR_{\sigma}(K, L) = \sum_{k, l} \dim_{\CC} \cH^{l} \left( \gr_{k} \DR \IC_{X}^{H}\right)_{0} K^{k} L^{l}.$$
	Here, $\cH^{l} (\gr_{k} \DR \IC_{X}^{H})$ carries an action of the torus $T$, hence equipped with a grading by $M$. We get a finite-dimensional vector space by taking the degree zero part. In \cite{KV:ICTV}*{Theorem 1.1}, we show that the generating function $\dR_{\sigma}(K, L)$ is completely determined by the topological data $H_{\sigma}$ by the following rule:
	$$ \dR_{\sigma}(K, L) = L^{-n} H_{\sigma}(LK^{-1/2}). $$
	Note that we have a factor of $q^{n}$ in the definition of $H_{\sigma}(q)$, and this differs from the notation $\widetilde{H}_{0, \sigma}(q)$ in \cite{KV:ICTV} exactly by a factor of $q^{n}$. From the result above, we see that
	$$ \cH^{l} \left(\gr_{k} \DR \IC_{X}^{H}\right)_{0} = 0$$
	unless $l + 2k = -n$ and $l < 0$, if $X$ is positive dimensional.
	
	We also mention that the Hodge structure of the cohomologies of the stalks $\cH^{j} \left(i_{x_{\sigma}}\sta \IC_{X}^{H}\right)$ are of Hodge--Tate type.
	
	\begin{lemm} \label{lemm:stalk-of-IC-pure-HT}
		$\cH^{j} (i_{x_{\sigma}}\sta \IC_{X}^{H})$ is pure of Hodge--Tate type with weight $j +n$. Also, if $X$ is positive dimensional, this is zero unless $-n \leq j < 0$.
	\end{lemm}
	\begin{proof}
		Consider the morphism $\pi \colon \widetilde{X} \to X$ given by a barycentric subdivision of $\sigma$ as in \cite{KV:ICTV}*{Remark 2.7}. Let $E = \pi^{-1}(x_{\sigma})$ be the inverse image of the torus fixed point. We point out that $E$ is an irreducible proper simplicial toric variety of dimension $n-1$. Note that $\IC_{X}^{H}$ is a summand of $\pi\lsta \QQ_{\widetilde{X}}^{H}[n]$ by the Decomposition theorem. By proper base change (see \cite{saito1990mixedHodgemodules}*{4.4.3}), we see that $i_{x_{\sigma}}\sta \IC_{X}^{H}$ is a direct summand of $\pi\lsta \QQ_{E}^{H}[n]$. Therefore, $\cH^{j}i_{x_{\sigma}}\sta \IC_{X}^{H}$ is a direct summand of $H^{n+j}(E, \QQ_{E})$ which is pure of Hodge--Tate type of weight $n+j$, by \cite{dCMM-toricmaps}*{Theorem A}. The second assertion follows from \cite{dCM-Decomposition}*{\S2.1.(12)}.
	\end{proof}
	
	\subsection{Shelling} \label{section:shelling}
	We introduce the concept of a \textit{shelling}. While the shelling is usually considered for polytopes, we use the language of cones, since it is better for our purposes.
	
	\begin{defi}
		Let $\sigma$ be a cone of dimension $n$. Let $\cP$ be the fan associated to $\sigma$, which is the collection of all faces of $\sigma$. A \textit{shelling} of $\sigma$ is a linear ordering $\mu_{1},\ldots, \mu_{s}$ of $\cP_{n-1}$ such that either $n = 1$, or it satisfies the following condition:
		\begin{enumerate}
			\item The set of facets $\cP_{n-2}^{\subset \mu_{1}}$ of the first facet $\mu_{1}$ has a shelling.
			\item For $1 < j \leq s$,
			$$ \mu_{j} \cap \left( \bigcup_{i=1}^{j-1} \mu_{i}\right) = \lambda_1 \cup \ldots \cup \lambda_r $$
			for some shelling $\lambda_1,\ldots, \lambda_r, \ldots, \lambda_t$ of $\cP_{n-2}^{\subset \mu_{j}}$.
		\end{enumerate}
		We say a cone is \textit{shellable} if it admits a shelling.
	\end{defi}
	
	By \cite{Bruggesser-Mani:Shellable}, all cones are shellable. Indeed, the shelling of a polytope of dimension $n-1$ obtained by a suitable hyperplane section of the cone provides a shelling of the cone itself.
	
	For a cone $\sigma$ of dimension $n$ and an integer $m \leq n$, we can construct a \textit{lexicographic shelling order} on the set
	$$ \Delta = \{ (\sigma, \lambda\lind{n-1},\ldots,\lambda\lind{m}) : \sigma \supset \lambda\lind{n-1} \supset \ldots\supset \lambda\lind{m}, \dim \lambda\lind{i} = i\}. $$
	First, fix a shelling $\prec_{\sigma}$ on $\sigma$. For each chain $\sigma \supset \ldots \supset \lambda\lind{l}$, we construct a shelling order $\prec_{\sigma,\ldots,\lambda\lind{l}}$ on $\lambda\lind{l}$ that satisfies the following property:
	\begin{enumerate}
		\item $$ \lambda\lind{l} \cap \left( \bigcup_{\widetilde{\lambda}\lind{l} \prec_{\sigma,\ldots,\lambda\lind{l+1}}\lambda\lind{l}} \widetilde{\lambda}\lind{l} \right) = \mu_1 \cup \ldots \cup \mu_r $$
		where $\mu_1,\ldots, \mu_r$ are the first $r$ facets of $\lambda\lind{l}$ of the shelling $\prec_{\sigma,\ldots, \lambda\lind{l}}$.
	\end{enumerate}
	We point out that this can be constructed inductively (in a decreasing manner) on $l$. Then we denote the lexicographic order in $\Delta$ by $\prec_{\mathrm{lex}}$. Pick two different chains $(\sigma, \lambda\lind{n-1},\ldots, \lambda\lind{m})$ and $(\sigma, \lambda\lind{n-1}',\ldots, \lambda\lind{m}')$ in $\Delta$ and assume that $k$ is the largest integer such that $\lambda\lind{k} \neq \lambda\lind{k}'$. Then $$ (\sigma, \lambda\lind{n-1},\ldots,\lambda\lind{m}) \prec_{\mathrm{lex}} (\sigma, \lambda\lind{n-1}', \ldots, \lambda\lind{m}') \quad \text{if and only if} \quad \lambda\lind{k} \prec_{\sigma,\ldots,\lambda\lind{k+1}} \lambda\lind{k}'. $$
	
	\section{Ishida complex} \label{section:Ishida-complex}
	This section is devoted to treat the Ishida complex \cite{Ishida,Ishida2} from scratch, in a self-contained manner. One of the goals is to show that the Grothendieck dual $\bfR\SheafHom_{X} (\duBois_{X}^{l}, \omega_{X})$ of the reflexive differentials can be expressed in terms of a very explicit, and combinatorial chain complex in terms of the structure sheaves of various torus-invariant subspaces. We mention that our proof does not use an injective resolution of the dualizing sheaf, as in \cite{Ishida2}.
	
	Throughout this section, we fix the affine toric variety $X$ associated to a cone $\sigma$ in $N$. For $\mu \subset \tau$ faces of $\sigma$ with $d_{\tau} = d_{\mu} + 1$, we denote by $n_{\mu,\tau}$ an element in $N$ such that $\langle \cdot, n_{\mu,\tau} \rangle : M \to \ZZ$ is zero on $\tau^{\perp}\cap M$ and maps $\tau\dual \cap \mu^{\perp}\cap M$ onto $\ZZ_{\geq 0}$. Note that this element is well-defined modulo $\langle\mu\rangle \cap N$. 
	
	\begin{defi} \label{defi:V-mu-l}
		For $\mu \in \cP$ and each $d_{\mu} \leq l \leq n$, we define the vector spaces
		$$ V_{\mu}^{l} := \bigwedge^{l-d_{\mu}} \mu^{\perp}, $$
		and the map between them $\varphi_{\mu,\tau}^{l}: V_\mu^l \to V_\tau^l$ defined by 
		$$\varphi_{\mu,\tau}^{l}(u_1 \wedge u_2 \wedge \ldots \wedge u_{l-d_\mu}) = \langle u_1, n_{\mu,\tau}\rangle u_2 \wedge \ldots \wedge u_{l-d_\mu}$$
		for $u_1 \in \mu^{\perp}$ and $u_2,\ldots, u_{l-d_{\mu}}\in \tau^{\perp}$. We will often write $\varphi_{\mu, \tau}$ instead of $\varphi_{\mu, \tau}^{l}$, if the index $l$ is clear from the context. By convention, $V_{\mu}^{l} = \RR$ (or $\RR_{\mu}$) if $d_{\mu}=l$. 
	\end{defi}
	We first show a lemma.
	
	\begin{lemm} \label{lemm:anti-commutativity-complex}
		Let $\mu \in \cP_{m}$ and $\tau \in \cP_{m+2}$ with $\mu \subset \tau$. Then there exist exactly two elements $\lambda_1$ and $\lambda_2$ in $\cP_{m+1}$ such that $\mu \subset \lambda_i \subset \tau$. Furthermore, we have
		$$ \varphi_{\lambda_1, \tau}^{l} \circ \varphi_{\mu,\lambda_1}^l + \varphi_{\lambda_2, \tau}^{l} \circ \varphi_{\mu,\lambda_2}^{l} = 0. $$
	\end{lemm}
	
	\begin{proof}
		The statement easily reduces to the 2-dimensional computation. Hence, assume $\mu = 0$ and that $\tau$ is spanned by $e_{1}$ and $e_{1} + t e_{2}$ for a basis $e_{1}, e_{2}$ of $\ZZ^{2}$ for some $t > 0$. Then $\lambda_{1}$ and $\lambda_{2}$ are spanned by $e_{1}$ and $e_{1}+ te_{2}$ respectively. It is then a direct calculation.
	\end{proof}
	
	\begin{defi} \label{defi:Ishida-complex}
		We define the \textit{$l$-th Ishida complex} $\Ish_{X}^{l}$ which is a complex of coherent sheaves on $X$. The Ishida complex is given by
		$$ \Ish_{X}^{l}: V_{0}^{l} \otimes_{\RR} \cO_{X} \to \bigoplus_{\mu \in \cP_{1}} V_{\mu}^{l} \otimes_{\RR} \cO_{S_{\mu}} \to \ldots \to \bigoplus_{\mu \in \cP_{l}} V_{\mu}^{l} \otimes_{\RR} \cO_{S_{\mu}}, $$
        living in cohomological degrees 0 to $l$. We describe the maps of the complex. If $\mu \in \cP_{m}$ and $\tau \in \cP_{m+1}$ with $\mu \subset \tau$, then the morphism $V_{\mu}^{l} \otimes \cO_{S_{\mu}} \to V_{\tau}^{l} \otimes \cO_{S_{\tau}}$ is given by the map $\varphi_{\mu,\tau}^{l} : V_{\mu}^{l} \to V_{\tau}^{l}$ and the restriction map $\cO_{S_{\mu}} \to \cO_{S_{\tau}}$. Otherwise, the map is zero. By Lemma \ref{lemm:anti-commutativity-complex}, we see that $\Ish_{X}^{l}$ is indeed a complex. 
	\end{defi}
	
	\begin{rema} \label{rema:things-btn-real-and-complex}
		We point out that the Ishida complex carries a natural grading by $M$. The graded pieces are complexes of finite dimensional complex vector spaces, whose terms are direct sums of $V_{\mu}^{l} \otimes_{\RR} \CC$. However, we are interested in the vanishing and non-vanishing of these complexes which makes the effect of the tensor by $\CC$ vacuous. Hence, we will sometime say the \textit{degree $u$-part} as the complex of real vectors spaces without explicitly mentioning the base change by $\CC$.
	\end{rema}
	For future references, we reserve the notation for the degree zero part of this complex.
	\begin{defi} \label{defi:Ishida-degree-0}
		We denote by
		$$ \Ish_{\sigma}^{l} : V_{0}^{l} \to \bigoplus_{\mu \in \cP_{1}} V_{\mu}^{l} \to \ldots \to \bigoplus_{\mu \in \cP_{l}} V_{\mu}^{l}.$$
		One can easily see that this is the degree 0 part of $\Ish_{X}^{l}$.
	\end{defi}

    Using this notation, we also describe the other graded pieces of the Ishida complex.
	\begin{lemm} \label{lemm:grade-parts-of-Ishida-complex}
		Let $u \in \tau_{\circ}\sta$ for some $\tau \in \cP$. Then the degree $u$-part of the Ishida complex $\Ish_{X}^{l}$ is isomorphic to
		$$ \bigoplus_{j=0}^{l} \bigwedge^{j} \tau^{\perp} \otimes \Ish_{\tau}^{l-j}  $$
		with the convention that $\Ish_{\tau}^{j}$ and $\bigwedge^{j} \tau^{\perp}$ is zero if $j > d_{\tau}$. Here, we view $\tau$ as a full-dimensional cone in $\langle \tau \rangle$.
	\end{lemm}
	\begin{proof}
		Note that $\cO_{S_{\mu}} = \Spec \CC[\sigma\dual\cap  \mu^{\perp}\cap M]$ has non-trivial degree $u$-part if and only if $\mu \subset \tau$ by Remark \ref{rema:grading-and-face-relations}. Therefore, the degree $u$-part of $\Ish_{X}^{l}$ is
		$$ V_{0}^{l} \to \bigoplus_{\mu \in \cP_{1}^{\subset \tau}} V_{\mu}^{l} \to \ldots \to \bigoplus_{\mu \in \cP_{l}^{\subset \tau}} V_{\mu}^{l}.$$
		Let $\overline{V} := V/\tau^{\perp} \simeq \Hom_{\RR}(\langle \tau \rangle , \RR)$ and fix a splitting $V \simeq \tau^{\perp} \oplus \overline{V}$. This splitting induces compatible splittings
		$$ \mu^{\perp} \simeq \tau^{\perp} \oplus \mu_{\tau}^{\perp}, $$
		where $\mu_{\tau}^{\perp} \subset \overline{V}$ is the orthogonal complement of $\mu$ viewed as a subspace of $\langle \tau \rangle$. Hence, for each $\mu$, we have the decomposition
		$$ V_{\mu}^{l} \simeq \bigwedge^{l - d_{\mu}} \mu^{\perp} \simeq \bigoplus_{j=0}^{l-d_{\mu}} \bigwedge^{j} \tau^{\perp} \otimes \bigwedge^{l- d_{\mu} - j} \mu_{\tau}^{\perp} \simeq \bigwedge^{j} \tau^{\perp} \otimes \overline{V}_{\mu}^{l -j}, $$
		where $\overline{V}_{\mu}^{l-j}$ is the vector space $V_{\mu}^{l-j}$ defined in Definition \ref{defi:V-mu-l}, but, by viewing $\mu$ as a subspace of $\langle \tau \rangle$. Since the splittings are compatible, we see that the degree $u$-part splits into the desired formula above.
	\end{proof}
	
	We first describe the cohomology of the Ishida complex at cohomological degree zero.
	\begin{prop} \label{prop:zeroth-coho-is-reflexivediff}
		The zero-th cohomology $\cH^{0}(\Ish_{X}^{l})$ of the Ishida complex is isomorphic to $\duBois_{X}^l$.
	\end{prop}
	\begin{proof}
		We use the identification $\Omega_{X}\ubind{l}(\log D) \simeq \bigwedge^{l} M \otimes_{\ZZ} \cO_{X}$ explained in \S\ref{section:diff-forms-on-toric}, where $D$ is the sum of all torus-invariant divisors. Since everything is torus equivariant, we exploit the grading by $M$. Therefore, it is enough to check when $u_{1}\wedge \ldots \wedge u_{l} \otimes \chi^{\alpha}$ is in $\duBois_{X}^l$ for $\alpha \in \sigma\dual \cap M$. Fix a torus-invariant divisor $S_{\rho}$ corresponding to a ray $\rho$. We see that $n_{0, \rho} \in N$ is the primitive element in $\rho$. One can assume that $u_{2},\ldots, u_{l} \in \rho^{\perp}$. If $u_{1} \in \rho^{\perp}$, then
		$$ d \log \chi^{u_1} \wedge \ldots \wedge d \log \chi^{u_{l}}$$
		is already in $\duBois_{X}^l$ in a neighborhood of $O_{\rho}$. If $u_{1} \notin \rho^{\perp}$, then we should require $\langle \alpha, n_{0, \rho} \rangle > 0$ in order to clear out the denominators. This exactly says that in order for the logarithmic differential form corresponding to $u_{1}\wedge \ldots \wedge u_{l} \otimes \chi^{\alpha}$ being inside $\duBois_{X}^l$ in a neighborhood of $O_{\rho}$, the vector $u_{1}\wedge \ldots \wedge u_{l}$ has to lie inside $\bigwedge^{l} \rho^{\perp}$ whenever $\alpha \in \rho^{\perp}$. This shows that $\duBois_{X}^l$ is exactly the kernel of the natural map
		$$ \bigwedge^{l} M_{\RR} \otimes \cO_{X} \to \bigoplus_{\rho \in \cP_{1}} V_{\rho}^{l} \otimes \cO_{S_{\rho}} $$
		by noticing that the kernel of $V_{0}^{l} \to V_{\rho}^{l}$ is exactly $\bigwedge^{l} \rho^{\perp}$.
	\end{proof}
	
	Given this, we can count the dimension of the graded pieces of $\duBois_{X}^l$ easily.
	\begin{lemm} \label{lemm:dimension-graded-of-reflex-diff}
		Let $u \in \tau_{\circ}\sta$. Then,
		$$ \dim_{\CC} \left( \duBois_{X}^l \right)_{u} = {n - d_\tau \choose l}.$$
	\end{lemm}
	\begin{proof}
		The degree $u$ part of $\duBois_{X}^l$ is exactly the kernel of the map $V_{0}^{l} \to \bigoplus_{\rho \in \cP_{1}^{\subset\tau}} V_{\rho}^{l}$. This equals $ \bigcap_{\rho \in \cP_{1}^{\subset\tau}} \bigwedge^{l} \rho^{\perp} = \bigwedge^{l} \tau^{\perp}$. Therefore, we get the formula above.
	\end{proof}
	
	The first step is to show that the $n$-th Ishida complex agrees with the dualizing sheaf $\omega_{X}$. Before that, we give some result on exactness of certain complexes.
	\begin{lemm} \label{lemm:exactness-for-Ish-n}
		Let $\sigma$ be a cone in $N \simeq \ZZ^{n}$. Let $\prec$ be a shelling order on $\sigma$. For each $\mu \in \cP_{n-1}$, we let $\widetilde{\cP}_{m}:= \{ \lambda \in \cP_{m} : \lambda \not\subset \bigcup_{\mu' \prec \mu} \mu' \}$.
		Then the complex
		$$ A_{\mu}^{\bullet} : \bigoplus_{\lambda\lind{0} \in \widetilde{\cP}_0} V_{\lambda\lind{0}}^{n} \to \bigoplus_{\lambda\lind{1} \in \widetilde{\cP}_1} V_{\lambda\lind{1}}^{n} \to \ldots \to \bigoplus_{\lambda\lind{n} \in \widetilde{\cP}_n} V_{\lambda\lind{n}}^{n} $$
		is exact. In particular, $\Ish_{\sigma}^{n}$ is exact by taking $\mu$ to be the minimal element of the shelling $\prec$.
	\end{lemm}
	\begin{proof}
		We show this by induction on dimension. If $n = 1$, the the complex is $V_{0}^{1} \to \RR$ which is an isomorphism. Hence, exact. For the inductive step, we use the shelling order to filter this complex, where the successive quotient can be expressed as a complex coming from one dimension lower. We express the shelling order of $\sigma$ as $\mu_{1}\prec\ldots \prec \mu_{N}$. We first assert that $A_{\mu}^{\bullet}$ is indeed a complex. This is because if $\tau \in \widetilde{\cP}_{m}$, we can easily see that $\lambda \in \widetilde{\cP}_{d_{\lambda}}$ if $\tau \subset \lambda$. Hence, for $\lambda \in \widetilde{\cP}_{m}$ and $\nu \in \widetilde{\cP}_{m+2}$, the two faces $\tau_1, \tau_2$ satisfying $\lambda\subset \tau_1, \tau_2 \subset \nu$ is also in $\widetilde{\cP}_{m+1}$. We also point out that
		$$ A_{\mu_1}^{\bullet} \supset \ldots \supset A_{\mu_{N}}^{\bullet}$$
		as complexes. Therefore, it is enough to show that $A_{\mu_N}^{\bullet}$ is exact, and the quotient $A_{\mu_k}^{\bullet}/ A_{\mu_{k+1}}^{\bullet}$ is exact for $k = 1, \ldots, N-1$. Note that $A_{\mu_N}$ is the complex
		$$ V_{\mu_N}^{n} \xrightarrow{\varphi_{\mu_N, \sigma}} V_{\sigma}^{n} $$
		sitting in cohomological degree $n-1$ and $n$. Both are of dimension 1 and $\varphi_{\mu_N, \sigma}$ is nonzero. Hence, the complex is exact. If $1\leq k \leq N-1$, we let
		$$ \mu_{k} \cap \left( \mu_{1}\cup\ldots \cup \mu_{k-1}\right) = \lambda_{1} \cup \ldots \cup \lambda_{r},$$
		where $\lambda_{1},\ldots, \lambda_{r}$ is the first $r$ elements of the shelling of $\mu_k$. If $k = 1$, we let $\lambda_{1},\ldots, \lambda_{r}$ to be a shelling of $\mu_1$. We denote $\widetilde{\cQ}_{m} := \{ \nu \in \cP_{m}^{\subset \mu_k} : \nu \not\subset \bigcup_{j=1}^{r} \lambda_{j} \}$ for $m = 1, \ldots, n-1$. Then we can describe the complex $A_{\mu_k}^{\bullet} / A_{\mu_{k+1}}$ as follows:
		$$ A_{\mu_k}^{\bullet} / A_{\mu_{k+1}}^{\bullet} : \bigoplus_{\nu\lind{0} \in \widetilde{\cQ}_0} V_{\nu\lind{0}}^{n} \to \bigoplus_{\nu\lind{1} \in \widetilde{\cQ}_1} V_{\nu\lind{1}}^{n} \to \ldots \to \bigoplus_{\nu\lind{n-1} \in \widetilde{\cQ}_{n-1}} V_{\nu\lind{n-1}}^{n}.$$
		Note that for each $\nu \in \widetilde{Q}_{m}$, we have
		$$ V_{\nu}^{n} = \bigwedge^{n-m} \nu^{\perp} \simeq \bigwedge^{n-m-1} \frac{\nu^{\perp}}{\mu_{k}^{\perp}} \otimes \mu_{k}^{\perp}.$$
		Notice that $\nu^{\perp} / \mu_{k}^{\perp}$ can be canonically identified as the orthogonal to $\nu$ as a cone sitting inside $\langle \mu_{k} \rangle$, via the canonical identification $\langle\mu_{k} \rangle\dual \simeq M
        _{\RR}/\mu_{k}^{\perp}$. Therefore, we can consider $A_{\mu_k}^{\bullet} / A_{\mu_{k+1}}^{\bullet}$ as the complex associated to the cone $\mu_k$ inside an $(n-1)$-dimensional vector space $\langle \mu_k \rangle$ and first several facets $\lambda_1,\ldots, \lambda_r$ of the shelling order of $\mu_k$. Hence, this is exact by the induction hypothesis.
	\end{proof}
	
	\begin{prop} \label{prop:top-Ishida-is-dualizing}
		The natural inclusion $\omega_{X} \to \Ish_{X}^{n}$ is a quasi-isomorphism of complexes.
	\end{prop}
	\begin{proof}
		Since $X$ is normal, $\omega_{X}$ is the sheaf of reflexive top differentials. Therefore $\cH^{0}(\Ish_{X}^{n}) \simeq \omega_{X}$. It remains to show that the higher cohomologies of $\Ish_{X}^{n}$ vanish. We exploit the grading by $M$. Fix a non-zero face $\tau \in \cP$ and $u \in \tau_{\circ}\sta$. Observing Remark \ref{rema:grading-and-face-relations}, we see that the degree $u$-part of the Ishida complex is exactly
		$$ V_{0}^{n} \to \bigoplus_{\lambda\lind{1}\in \cP_{1}^{\subset \tau}} V_{\lambda\lind{1}}^{n} \to \ldots \to \bigoplus_{\lambda\lind{d_{\tau}-1} \in \cP_{d_{\tau} -1}^{\subset \tau}} V_{\lambda\lind{d_{\tau}-1}}^{n} \to V_{\tau}^{n}. $$
		We see that $V_{\lambda\lind{i}}^{n} \simeq V_{\tau}^{n} \otimes \bigwedge^{n-d_{\tau} - i} \lambda_{(i)}^{\perp}/ \tau^{\perp}$.
		By Lemma \ref{lemm:exactness-for-Ish-n} applied to $\tau$, we see that this is an exact sequence. For $u \in 0_{\circ}\sta$, the degree $u$-part of the Ishida complex is $V_{0}^{n}$ concentrated at cohomological degree 0.
	\end{proof}
	
	We give with the main result of this section. Namely, the Ishida complex agrees with the Grothendieck dual of the sheaf of reflexive differentials.
	
	\begin{theo} \label{theo:Gro-dual-of-Ishida-is-duBois}
		$\bfR \SheafHom_{\cO_{X}} (\duBois_{X}^l , \omega_{X}) \simeq \Ish_{X}^{n-l}$.
	\end{theo}
	
	\begin{proof}
		It is enough to show that $\bfR \SheafHom (\Ish_{X}^{n-l}, \omega_{X}) \simeq \duBois_{X}^l$ since $X$ is Cohen--Macaulay. First, we consider the hypercohomology spectral sequence. We get
		$$ E_{1}^{i,j} = \bigoplus_{\tau \in \cP_{-i}} \SheafExt_{\cO_{X}}^{j} (\cO_{S_{\tau}}, \omega_{X}) \otimes (V_{\tau}^{n-l})\dual \implies \SheafExt_{\cO_{X}}^{i+j} (\Ish_{X}^{n-l}, \omega_{X}). $$
		Note that $S_{\tau}$ are Cohen--Macaulay, and therefore
		$$ \SheafExt^{j}_{\cO_{X}}(\cO_{S_{\tau}}, \omega_{X}) = \SheafExt^{j - d_{\tau}}_{\cO_{S_{\tau}}} (\cO_{S_{\tau}}, \omega_{S_{\tau}}) = \begin{cases} \omega_{S_{\tau}} & j = d_{\tau} \\ 0 & j \neq d_{\tau}. \end{cases}$$
		Therefore, $E_{1}^{i,j} = 0$ unless $i + j  =0$. Hence, the spectral sequence degenerates at $E_{1}$ and $\SheafExt_{\cO_{X}}^{i}(\Ish_{X}^{p}, \omega_{X}) = 0$ for $i \neq 0$. Moreover, we count the dimension of the graded pieces in each degree. First, note that $(\omega_{S_{\tau}})_{u}$ has rank 1 for $u \in \tau_{\circ}\sta$ and zero otherwise. Also, the spectral sequence induces a filtration on $\SheafExt_{\cO_{X}}^{0}(\Ish_{X}^{n-l}, \omega_{X})$ whose successive quotients are isomorphic to
		$$ \bigoplus_{\tau \in \cP_{m}} \omega_{S_{\tau}} \otimes (V_{\tau}^{n-l})\dual.$$
		Therefore,
		$$ \dim_{\CC} \SheafExt_{\cO_{X}}^{0}(\Ish_{X}^{n-l}, \omega_{X})_{u} = \dim_{\RR} (V_{\tau}^{n-l})\dual = {n-d_{\tau} \choose l}, \qquad \text{for } u \in \tau_{\circ}\sta.$$
		
		Finally, for $a \cdot \chi^{u} \in \duBois_{X}^l$, we construct a morphism $\varphi : \Ish_{X}^{n-l} \to \Ish_{X}^{n}$. The $i$-th term is defined for $0\leq i \leq n-l$ as follows:
		$$\varphi^{i} : \bigoplus_{\mu \in\cP_{i}} V_{\mu}^{n-l} \otimes \cO_{S_{\mu}} \to \bigoplus_{\mu \in\cP_{i}} V_{\mu}^{n} \otimes \cO_{S_{\mu}}$$
		by $\bullet \wedge a : V_{\mu}^{i} \to V_{\mu}^{n}$ and the multiplication by $\chi^{u}$ on $\cO_{S_{\mu}}$. We point out that the multiplication $\bullet \wedge a$ does not make sense if $a$ is not in $\bigwedge^{l} \mu^{\perp}$, but in this case, we have $u \notin \tau^{\perp} \cap \sigma\dual \cap M$ (by Proposition \ref{prop:zeroth-coho-is-reflexivediff}) which makes the morphism zero anyways. It is clear that $\varphi$ is actually a morphism of complexes, and induces an injective homomorphism
		$$ \duBois_{X}^l \to \bfR^{0}\SheafHom (\Ish_{X}^{n-l},\Ish_{X}^{n}). $$
		Again, using the grading by $M$ and Proposition \ref{prop:top-Ishida-is-dualizing}, we see that this homomorphism is bijective since the dimensions of the graded pieces match up by Lemma \ref{lemm:dimension-graded-of-reflex-diff}.
	\end{proof}

    \begin{rema}
        We can also see that the isomorphism 
        $$\bfR\SheafHom_{\cO_{X}}(\Ish_{X}^{n-l}, \omega_{X}) \simeq \bfR^{0} \SheafHom_{\cO_{X}}(\Ish_{X}^{n-l}, \omega_{X}) \simeq \duBois_{X}^{l}$$ holds for arbitrary toric varieties $X$. This is because the restrictions of the homomorphisms $a \cdot \chi^{u} : \Ish_{X}^{n-l} \to \Ish_{X}^{n}$ on affine charts are compatible with the restriction of differential forms. Hence, we get a global morphism $\duBois_{X}^{l} \to \bfR^{0}\SheafHom_{\cO_{X}}(\Ish_{X}^{n-l}, \Ish_{X}^{n})$ and verifying the isomorphism can be done affine locally.
    \end{rema}
	
	We now simply rephrase Theorem \ref{theo:MP-lcd-intermsof-DuBois} in terms of the Ishida complex.
	\begin{prop} \label{prop:lcd-intermsof-Ishida}
		Let $X$ be a toric variety of dimension $n$. Then the following are equivalent:
		\begin{enumerate}
			\item $\lcdef(X) \leq c$,
			\item $\cH^{j+l+1}(\Ish_{X}^{n-l}) = 0$ for all $j \geq c$ and $l \geq 0$.
		\end{enumerate}
	\end{prop}
	
	Here, we point out that $\omega_{X}^{\bullet}[-\dim X] = \omega_{X}$ since $X$ is Cohen--Macaulay. We end this section with a linear algebra lemma which will be used later.
	\begin{lemm} \label{lemm:exact-sequence-from-simplicial}
		Let $\mu \in \cP_{m}$ such that the corresponding torus-invariant subvariety $S_{\mu}$ is simplicial. Also assume that $m > 0$. Then the sequence
		$$ V_{\mu}^{l} \to \bigoplus_{\lambda\lind{m+1} \in \cP_{m+1}^{\supset \mu}} V_{\lambda\lind{m+1}}^{l} \to \ldots \to \bigoplus_{\lambda\lind{l} \in \cP_{l}^{\supset \mu}} V_{\lambda\lind{l}}^{l} $$
		is exact for $l \geq m$.
	\end{lemm}
	\begin{proof}
		Let $\overline{N} = N/ \langle \mu \rangle \cap N$ and we have a natural identification $\mu^{\perp} \simeq \Hom_{\RR}(\overline{N}_{\RR} , \RR)$. This gives an identification $\nu^{\perp} \simeq \overline{\nu}^{\perp} \subset \Hom_{\RR}(\overline{N}_{\RR}, \RR)$ where $\overline{\nu}$ is the image of $\nu$ in $\overline{N}_{\RR}$. Let $\overline{\cP}$ be the collection of faces of the cone $\overline{\sigma}$. Then we have a natural 1 to 1 correspondence between $\cP_{m+k}^{\supset \mu}$ and $\overline{\cP}_{k}$ by taking $\lambda\lind{m+k}$ to $\overline{\lambda\lind{m+k}}$. Under this identification, one can see that the complex above is isomorphic to $\Ish_{\overline{\sigma}}^{l}$. Note that the toric variety corresponding to $\overline{\sigma} \subset \overline{N}$ is $S_{\mu}$ and this is a simplicial toric variety, hence has quotient singularities. Following the proofs in \cite{Mustata-Popa22:Hodge-filtration-local-cohomology}*{Corollary 4.29}, we have $$\SheafExt_{\cO_{S_{\mu}}}^{i}(\duBois_{S_{\mu}}^{n-m-l}, \omega_{S_{\tau}}) \simeq \begin{cases} \duBois_{S_{\mu}}^l & i = 0 \\ 0 & i \neq 0. \end{cases}$$
		The cohomologies of the sequence above are exactly the degree zero part of the $\SheafExt$ sheaves above. Since we know that $\left(\duBois_{S_{\mu}}^{l}\right)_{0} = 0$ for $l > 0$, we are done.
	\end{proof}
	
	\section{The Hodge module structure of $\QQ_{X}^{H'}$}\label{section:Hodge-module-structure-Q}

    \subsection{Proof of Theorems \ref{theo:MHM-structure} and \ref{theo:depth-of-Du-Bois-general}}
	We describe the structure of $\QQ_{X}^{H'}$ as in Theorem \ref{theo:MHM-structure}. After that, we prove Theorem \ref{theo:depth-of-Du-Bois-general} on the depth of the reflexive differentials. All of the assertions are local, hence we assume that $X$ is an $n$-dimensional affine toric variety corresponding to a cone $\sigma$.
    
    First, we focus on the claim in Theorem \ref{theo:MHM-structure} that $\QQ_{X}^{H'}$ has no cohomologies outside degrees $[-(n-3), 0]$, which is the same as claiming that $\lcdef(X) \leq n-3$. To prove this, we can directly use the result of \cite{Dao-Takagi} since toric varieties are Cohen--Macaulay or we can also appeal to Theorem \ref{theo:simple-upperbound-depth} since $\sigma$ is simple in dimension $n-2$. However, we give an alternate self contained proof here using the criteria in Theorem \ref{theo:MP-lcd-intermsof-DuBois}.
    
    From the description of the Ishida complex, it is clear that $\SheafExt_{\cO_{X}}^{p}(\duBois_{X}^{q} , \omega_{X}) = 0$ if $p + q > n$, since $\Ish_{X}^{n-q}$ sits inside cohomological degrees $0$ to $n-q$. 
    Also, for $q = 0$, we have $\duBois_{X}^{0} \simeq \cO_{X}$ and so, all of the higher $\SheafExt$ sheaves vanish. It only remains to prove that $\SheafExt_{\cO_{X}}^{n-1}(\duBois_{X}^{1}, \omega_{X}) = 0$ which we show now. In fact, the proof of Theorem \ref{theo:simple-upperbound-depth} is a direct generalization of the following argument. We consider the Ishida complex:
	$$ \Ish_{X}^{n-1} : V_{0}^{n-1} \otimes \cO_{X} \to \ldots \to \ldots \bigoplus_{\lambda \in \cP_{n-2}} \lambda^{\perp} \otimes \cO_{S_{\lambda}} \to \bigoplus_{\tau \in \cP_{n-1}} \RR_{\tau} \otimes \cO_{S_{\tau}}$$
	We use the grading by $M$ and compute the cohomology at the last slot. The only relevant degrees are $u \in \tau_{\circ}\sta$ for $\tau \in \cP_{n-1}$ and $u \in \sigma_{\circ}\sta$, which is the case when $u = 0$. For the first case $u \in \tau_{\circ}\sta$, the last two entries of $(\Ish_{X}^{n-1})_{u}$ is
	$$ \bigoplus_{\lambda \in \cP_{n-2}^{\subset \tau}} \lambda^{\perp} \to \RR_{\tau}. $$
	The set $\cP_{n-2}^{\subset \tau}$ is non-empty and therefore the map is surjective. We consider the case $u = 0$. Note that for each $\lambda \in \cP_{n-2}$, there are exactly two faces $\tau , \tau'\in \cP_{n-1}$ such that $\tau \supset \lambda$. Using Lemma \ref{lemm:exact-sequence-from-simplicial} applied to $\lambda$, we see that we have an isomorphism
	$$ \lambda^{\perp} \xrightarrow{(\varphi_{\lambda, \tau},\varphi_{\lambda, \tau'}) } \RR_{\tau} \oplus \RR_{\tau'}.$$
	Therefore, we have a vector $v_{\tau} \in\lambda^{\perp}$ that maps zero to $\RR_{\tau'}$ and non-zero to $\RR_{\tau}$. This shows that the map $\bigoplus_{\lambda \in \cP_{n-2}} \lambda^{\perp} \to \bigoplus_{\tau \in \cP_{n-1}} \RR_{\tau}$ is surjective.

    Now, we give a proof of Theorem \ref{theo:MHM-structure}.
    
    \begin{proof}[Proof of Theorem \ref{theo:MHM-structure}]
        The structure of the proof goes as follows. First, we look at the behavior of $\QQ_{X}^{H'}$ near the positive dimensional torus orbits, and use induction on dimension to get a control on the structure of $\QQ_{X}^{H'}$ away from the torus fixed point. Then, we determine the behavior at the torus fixed point by calculating the contributions of various intersection cohomologies coming from the positive dimensional torus orbits.
        
        \textbf{Step 1.} For $\tau$ a proper face of $\sigma$, we analyze the structure of $\QQ_{X}^{H'}$ on the open chart $U_{\tau}$. From \S\ref{section:toric-var-basic}, we see that $U_{\tau}$ can be written as $V_{\tau} \times O_{\tau}$ where $V_{\tau}$ is the full-dimensional toric variety corresponding to $\tau$. Let $\pi : U_{\tau} \to V_{\tau}$ be the smooth projection of relative dimension $n-d_{\tau}$. Note that $\QQ_{U_{\tau}}^{H} = \pi\sta \QQ_{V_{\tau}}^{H}$ and therefore, $\pi\sta \QQ_{V_{\tau}}^{H'} \simeq \QQ_{U_{\tau}}^{H'}[-(n-d_{\tau})]$. Hence, we get
		$$ \pi\sta \gr_{d_{\tau}-k}^{W} \cH^{-l}\QQ_{V_{\tau}}^{H'} \simeq \left(\gr^{W}_{n-k} \cH^{-l+(n-d_{\tau})} \QQ_{U_{\tau}}^{H'}[-(n-d_{\tau})]\right) [n-d_{\tau}] \simeq\left(\gr_{n-k}^{W} \cH^{-l}\QQ_{U_{\tau}}^{H'} \right) [n-d_{\tau}].$$ 
		Consider the decomposition by strict support of $\gr_{d_{\tau}-k}^{W} \cH^{-l}\QQ_{V_{\tau}}^{H'}$. Denote by $\cM_{\tau}$ the summands with strict support $x_{\tau} \in V_{\tau}$, the torus fixed point of $V_\tau$. By induction, we see that 
		$$\cM_{\tau} \simeq \QQ_{x_{\tau}}^{H} (-j)^{a_{\tau}^{l, j}},$$
		where $k + 2j = d_{\tau}$. Implicitly, this means that $\cM_{\tau}$ is zero if $d_{\tau} - k$ is odd. From this, we are able to compute the decomposition by strict support of the module $\gr_{n-k}^{W} \cH^{-l} \QQ_{U_{\tau}}^{H'}$. 

        In particular, the term with strict support $O_{\tau}$ is $\QQ_{O_{\tau}}^{H'}(-j)^{a_{\tau}^{l, j}}$, where $k + 2j = d_{\tau}$. Also, we see that $a_{\tau}^{l, j}$ is zero unless $j \geq 1$. Since we know that components with strict support on each orbit $O_{\tau}$, we see that
		\begin{equation} \label{eqn:decomp-without-torusfixedpoint}
		    \gr_{n-k}^{W} \cH^{-l} \QQ_{X}^{H'} \simeq \cK_{k, l} \oplus\bigoplus_{\substack{j \geq 1 \\ k + 2j < n}} \bigoplus_{\lambda \in \cP_{k + 2j}}  \IC_{S_{\lambda}}^{H}(-j)^{a_{\lambda}^{l, j}}
		\end{equation}
		where $\cK_{k, l}$ is a pure Hodge module of weight $n-k$ supported on the torus fixed point $x_{\sigma}$, and we view it as a pure Hodge structure. This also shows that the numbers $a_{\lambda}^{l,j}$ only depend on $\lambda$.

        It remains to show the two following things:
        \begin{enumerate}
			\item $\cK_{k, l} = 0$ unless $l+1\leq k \leq n-2$, and
			\item $\cK_{k, l}$ is pure of Hodge--Tate type if it is non-zero.
		\end{enumerate}

        By \cite{Popa-Park:lefschetz}*{Proposition 6.4}, we see that $\cK_{k, l} = 0$ unless $k \geq l+1$ since $\gr_{n-k}^{W} \cH^{-l} \QQ_{X}^{H'} = 0$ for $l > 0$ and $n-k \geq n-l$. For $k=l=0$, we can directly see that $\gr_{n}^{W} \cH^{0} \QQ_{X}^{H'} \simeq \IC_{X}^{H}$. Hence, it remains to show that $\cK_{k, l}$ is of weight $\geq 2$ (i.e., $k \leq n-2$) and of Hodge--Tate type.

        For this, we will use two spectral sequences
        $$ \label{equation:1-spec-seq}
            E_{2}^{p,q} = \cH^{p} i_{x_{\sigma}}\sta \cH^{q}\QQ_{X}^{H'} \implies \cH^{p+q} i_{x_{\sigma}}\sta \QQ_{X}^{H'} = \begin{cases}
            \QQ_{x_{\sigma}}^{H} & p + q = -n \\
            0 & \text{otherwise},
        \end{cases}
        $$
        and
        $$
           \widetilde{E}_{1}^{r,s}(q) = \cH^{r+s} i_{x_{\sigma}}\sta \gr_{-r}^{W} \cH^{q} \QQ_{X}^{H'} \implies \cH^{r+s} i_{x_{\sigma}}\sta \cH^{q} \QQ_{X}^{H'}. $$

        The aim for the rest of the proof is to show that $\widetilde{E}_{1}^{r, -r}(q)$ is mixed of Hodge--Tate type of weight $\geq 2$ for all $r,q$. This would imply the same for $\cK_{k, l}$ since it is a direct summand of $\widetilde{E}_{1}^{k-n, -(k-n)}(-l)$, by (\ref{eqn:decomp-without-torusfixedpoint}). Observe that the property `mixed of Hodge--Tate type of weight $\geq 2$' is preserved under taking subquotients and extensions.
           
        \textbf{Step 2.} Here, we show that $\cH^{p} i_{x_{\sigma}}\sta \cH^{q}\QQ_{X}^{H'}$ is mixed of Hodge--Tate type of weight $\geq 2$ for $-n<p<0$. We use the decomposition in (\ref{eqn:decomp-without-torusfixedpoint}) to understand the term $\widetilde{E}_{1}^{r,s}(q)$ of the second spectral sequence. Observe that if $S_{\lambda}$ is of dimension $-r-2j$, then Lemma \ref{lemm:stalk-of-IC-pure-HT} implies that $\cH^p i_{x_\sigma}^* \IC^H_{S_\lambda}(-j)$ is pure of Hodge--Tate type of weight $p-r$ and Lemma \ref{lemm:stalk-of-IC-pure-HT} implies that $\cH^p i_{x_\sigma}^* \IC^H_{S_\lambda}(-j)$ is zero unless $0 \leq p-r-2j \leq \frac{-r-2j}{2}$.
        In particular, if $j \geq 1$, we have that the weight $p-r \geq 2$. This tells us that $\widetilde{E}_{1}^{r,p-r}(q)$ is pure of Hodge--Tate type of weight $\geq 2$ if $p < 0$ and $q<0$, since we do not have a contribution from $\cK_{k, l}$'s that are supported at $x_{\sigma}$, or from $\IC^H_X$. Therefore, the second spectral sequence tells us that $\cH^{p}i_{x_{\sigma}}\sta \cH^{q}\QQ_{X}^{H'}$ is mixed of Hodge--Tate type of weight $\geq 2$ since it admits a filtration whose graded pieces are subquotients of $\widetilde{E}_{1}^{r,p-r}(q)$. 
        
        When $p<0$ and $q=0$, the same argument shows that $\widetilde{E}_{1}^{r,p-r}(0)$ is pure of Hodge--Tate type of weight $\geq 2$ for $r \neq -n$. When $r=-n$, we have from Lemma \ref{lemm:stalk-of-IC-pure-HT} that $\widetilde{E}_{1}^{-n,p+n}(0)$ is pure of Hodge--Tate type of weight $\geq 2$ for $p \neq -n$. Therefore $\cH^{p}i_{x_{\sigma}}\sta \cH^{0}\QQ_{X}^{H'}$ is mixed of Hodge--Tate type of weight $\geq 2$ for $-n<p<0$.

        \textbf{Step 3.} From the first spectral sequence, we see that for $-(n-3) \leq q \leq 0$, we get $E_{\infty}^{0, q} = 0$. Additionally, observe that $E_2^{p,q} = \cH^{p} i_{x_{\sigma}}\sta \cH^{q} \QQ_{X}^{H'} = 0$ for $p>0$ since $\widetilde{E}_{1}^{r, p-r}(q) = 0$ for $p>0$ by Lemma \ref{lemm:stalk-of-IC-pure-HT}. Therefore, $E_2^{0,q} = \cH^{0} i_{x_{\sigma}}\sta \cH^{q} \QQ_{X}^{H'}$ admits a filtration whose successive quotients are subquotients of $E_2^{p',q'} = \cH^{p'} i_{x_{\sigma}}\sta \cH^{q'} \QQ_{X}^{H'}$ for $p' + q' = q-1$ and $p'<-1$. Since any such $\cH^{p'} i_{x_{\sigma}}\sta \cH^{q'} \QQ_{X}^{H'}$ is mixed of Hodge--Tate type of weight $\geq 2$ by Step 2, so is $\cH^{0} i_{x_{\sigma}}\sta \cH^{q} \QQ_{X}^{H'}$.

        \textbf{Step 4.} Finally, we go back to the second spectral sequence again. Since $\cH^{0}i_{x_{\sigma}}\sta \cH^{q} \QQ_{X}^{H'}$ is mixed of Hodge--Tate type with weights $\geq 2$, so are $\widetilde{E}_{\infty}^{r, -r}(q)$. First, note that $\widetilde{E}_{1}^{r, s}(q) = 0$ for $r + s > 0$ by Lemma \ref{lemm:stalk-of-IC-pure-HT}. This implies that $\widetilde{E}_{\infty}^{r, -r}(q)$ is a quotient of $\widetilde{E}_{1}^{r, -r}(q)$ by an object which admits a filtration whose successive quotients are subquotients of $\widetilde{E}_{1}^{r', s'}(q)$ for $r' + s' = -1$. However, by the same argument in Step 2, we see that $\widetilde{E}_{1}^{r', s'}(q)$ is mixed of Hodge--Tate type of weight $\geq 2$. This implies that $\widetilde{E}_{1}^{r, -r}(q)$ is as well. This concludes the proof since $\cK_{k, l}$'s are direct summands of these.
    \end{proof}
	
	Using the structure theorem for $\QQ_{X}^{H'}$, we prove Theorem \ref{theo:depth-of-Du-Bois-general}.
	
	\begin{proof}[Proof of Theorem \ref{theo:depth-of-Du-Bois-general}]
		From the realization of $\bfR\SheafHom (\duBois_{X}^{k} ,\omega_{X})$ in terms of the Ishida complex, we immediately see that $\SheafExt_{\cO_{X}}^{l}(\duBois_{X}^{k}, \omega_{X}) = 0$ if $l+k > n$. Therefore, we only show $\SheafExt_{\cO_{X}}^{n-k}(\duBois_{X}^{k}, \omega_{X}) = 0$ for $k \leq n/2$. By exploiting the grading by $M$, it is enough to show that $H^{n-k}(\Ish_{X}^{n-k})_{u} = 0$ for $u \in \tau_{\circ}\sta$, where $\tau$ is a face of $\sigma$. From Lemma \ref{lemm:grade-parts-of-Ishida-complex}, we see that
		$$ \left(\Ish_{X}^{n-k}\right)_{u} \simeq \bigoplus_{j=0}^{n-k} \bigwedge^{j} \tau^{\perp} \otimes \Ish_{\tau}^{n-k-j}. $$
		By taking the $(n-k)$-th cohomology, we see that $H^{n-k}(\Ish_{X}^{n-k})_{u} \simeq H^{n-k}\Ish_{\tau}^{n-k}$. Note that $n-k = d_{\tau} - (k-n+d_{\tau})$ and $k - n + d_{\tau} \leq -\frac{n}{2} + d_{\tau} \leq \frac{d_{\tau}}{2}$ since $d_{\tau} \leq n$. Using induction on dimension, we only have to show that $H^{n-k}(\Ish_{\sigma}^{n-k}) = 0$ for $k \leq n/2$, which is the graded piece at degree 0.
		We do this by considering the spectral sequence
		$$ E_{2}^{p,q} = \cH^{p} \left( \gr_{k} \DR \cH^{q} \DD \QQ_{X}^{H'} \right) \implies \SheafExt_{\cO_{X}}^{k+p+q}(\duBois_{X}^{k}, \omega_{X}).$$
		Hence, it is enough to show $\cH^{p}\left( \gr_{k} \DR \cH^{q} \DD \QQ_{X}^{H'}\right)_{0} = 0$ for $p+ q = n-2k$. 

        Dualizing Theorem \ref{theo:MHM-structure}(2) and \ref{theo:MHM-structure}(3), we have
        \begin{align*}
            \gr_{-n}^{W} \cH^{0} \DD\QQ_{X}^{H'} &\simeq \gr_{k}\DR \IC_{X}^{H}(n),\\
            \gr_{-n+t}^{W} \cH^{q} \DD \QQ_{X}^{H'} &\simeq \bigoplus_{j \geq 1} \bigoplus_{\lambda \in \cP_{t + 2j}} \IC_{S_{\lambda}}^{H} (n-t-j)^{a_{\lambda}^{q, j}}, \qquad \text{for } (q, t)\neq (0,0).
        \end{align*}
        and the second term is zero unless $q+1 \leq t \leq n-2$. Therefore, we get
        \begin{align*}
            \cH^{n-2k} \gr_{k} \DR \gr_{-n}^{W} \cH^{0} \DD\QQ_{X}^{H'} &\simeq \cH^{n-2k} \gr_{k-n}\DR \IC_{X}^{H},\\
            \cH^{n-2k-q} \gr_{k} \DR \gr_{-n+t}^{W} \cH^{q} \DD \QQ_{X}^{H'} &\simeq \bigoplus_{j \geq 1} \bigoplus_{\lambda \in \cP_{t + 2j}} \left(\cH^{n-2k-q} \gr_{k -n+t+j} \DR \IC_{S_{\lambda}}^{H} \right)^{a_{\lambda}^{q,j}}.
        \end{align*}
        As mentioned in \S\ref{sec:ICTV-summary}, $(\cH^{n-2k} \gr_{k-n}\DR \IC_{X}^{H})_0$ is non-zero only when $n-2k<0$, which cannot happen since $k \leq n/2$. Similarly, $\left( \cH^{n-2k-q} \gr_{k-n+t+j} \DR \IC_{S_{\lambda}}^{H}\right)_{0}$ is non-zero only when
		$$ n - 2k - q + 2 (k-n+t+j) = - \dim S_{\lambda} = 2j-n+t.$$
        However, this implies that $q = t$, which cannot happen since $t \geq q+1$. Therefore, we have the required vanishing by considering the spectral sequence associated to the weight filtration.
	\end{proof}

	\begin{rema}
		One can see from this fact that $H^{n-k}(\Ish_{\sigma}^{n-k}) = 0$ for $k \leq n/2$, saying that the last map of this complex is surjective. For $k = 0, 1$, we have a direct combinatorial proof of this fact as in the beginning of this section. However, we were not able to find a direct combinatorial proof for other $k$'s. We remark that it would be interesting if one can find such a combinatorial proof.
	\end{rema}

    \begin{rema}
        Our proof that the last map of the complex $\Ish^{n-k}_{\sigma}$ is surjective for $k \leq n/2$ only works for a \textit{rational} polyhedral cone $\sigma$. It would be interesting to see if the same statement holds for \textit{non-rational} cones as well. More precisely, it is possible to construct some `Ishida-like' complex for non-rational cones and we ask: in this non-rational case, would the same surjectivity statement hold?
    \end{rema}
        
	\subsection{Hodge structure on singular cohomology of proper toric varieties}
    Here, we give a quick proof that the Hodge structure on the singular cohomology of a proper toric variety is mixed of Hodge--Tate type. Before that, we show that the intersection cohomology of a proper toric variety is pure of Hodge--Tate type. This is implicitly written in \cite{saito-toric}*{(1.7.6)} in the projective case, but we give a proof for completeness.
    \begin{lemm}
        Let $X$ be a proper toric variety of dimension $n$. The intersection cohomology
        $$ \IH^{i}(X) := \HH^{i-n}(X, \IC_{X}^{H})$$
        is pure of Hodge--Tate type, of weight $i$.
    \end{lemm}
    \begin{proof}
        Let $\pi \colon \widetilde{X} \to X$ be a toric resolution of singularities. Then $\QQ_{\widetilde{X}}^{H'}$ is a pure Hodge module, and $\IC_{X}^{H}$ is a summand of $\pi\lsta \QQ_{\widetilde{X}}^{H'}$ by the Decomposition theorem. Therefore, $\HH^{i-n}(X, \IC_{X}^{H})$ is a summand of $H^{i}(\widetilde{X}, \QQ)$ which is pure of Hodge--Tate type, with weight $i$.
    \end{proof}
    
    We next show that the singular cohomology groups of proper toric varieties are mixed of Hodge--Tate type.
    
    \begin{proof}[Proof of Theorem \ref{coro:sing-coho-HT-type}]
        Let $a : X \to \{ \mathrm{pt}\}$ be the morphism to a point. Note that $H^{k}(X, \QQ) \simeq \cH^{k}(a\lsta \QQ_{X}^{H})$. We use the spectral sequence
        $$ E_{2}^{p,q} = \cH^{p} \left( a\lsta \cH^{q} \QQ_{X}^{H}\right) \implies \cH^{p+q} (a\lsta \QQ_{X}^{H}). $$
        Therefore, it is enough to show that each term in the $E_{2}$-page is mixed of Hodge--Tate type. For this, we use another spectral sequence
        $$ E_{1}^{r, s} = \cH^{r+s} \left( a\lsta \gr_{-r}^{W} \cH^{q} \QQ_{X}^{H}\right) \implies \cH^{r+s} (a\lsta \cH^{q} \QQ_{X}^{H}).$$
        Note that $E_{1}^{r, s}$ is in fact pure of Hodge--Tate type of weight $s$. The purity follows from Saito's Decomposition theorem, and the Hodge--Tate type property follows from Theorem \ref{theo:MHM-structure} and the fact that the intersection cohomology of a proper toric variety is of Hodge--Tate type.
    \end{proof}

    Furthermore, in \cite{LCDTV2}, we describe some techniques for computing the Hodge--Du Bois numbers of toric varieties. Moreover, for a projective toric variety with an ample $\QQ$-line bundle $L$, we relate the Lefschetz morphism $c_{1}(L) \cup (\cdot)$ on singular cohomology with the Ishida complex of an affine toric variety of one-dimensional higher. We state the relevant result used in this article and refer the reader to \cite{LCDTV2} for further details.

    \begin{theo} \cite{LCDTV2}*{Theorem 1.4} \label{theo:Lefschetz}
        Let $X$ be an $n$-dimensional affine toric variety associated to a full-dimensional cone $\sigma$. Let $\rho$ be a rational ray in the interior of $\sigma$ and consider the toric morphism $\pi \colon \widetilde{X} \to X$ corresponding to inserting the ray $\rho$ in $\sigma$. Let $E$ be the projective toric variety given by the inverse image of the torus fixed point. Then we have the long exact sequence
        $$ \ldots \to  \HH^{i-1}(E, \Ish_{E}^{l-1}) \xrightarrow{c_{1}\dual} \HH^{i}(E, \Ish_{E}^{l}) \to H^{i}(\Ish_{\sigma}^{l}) \to \HH^{i}(E, \Ish_{E}^{l-1}) \xrightarrow{c_{1}\dual} \ldots $$
        where $c_{1}\dual \colon \HH^{i-1}(E, \Ish_{E}^{l-1}) \to \HH^{i}(E,\Ish_{E}^{l})$ is dual to $c_1:H^{n-i-1}(E,\duBois^{n-l-1}_E) \to H^{n-i}(E, \duBois^{n-l}_E)$, the Chern class map of the $\QQ$-Cartier divisor class $(-E)|_{E}$ on $E$.
    \end{theo}

        \section{Cones over simplicial polytopes}\label{section:cone-over-simplicial-polytope}
    In this section, we consider the affine toric varieties corresponding to cones over simplicial polytopes. This is precisely the case when the toric variety has an isolated non-simplicial locus.

    \subsection{Structure of $\QQ_{X}^{H'}$}
    We prove Theorem \ref{theo:cone-over-simplicial-polytope} and Theorem \ref{theo:cone-over-simplicial} simultaneously. This fully describes the numbers appearing in the structure theorem, Theorem \ref{theo:MHM-structure} for the above class of toric varieties.
    
    \begin{proof}[Proof of Theorem \ref{theo:cone-over-simplicial-polytope} and Theorem \ref{theo:cone-over-simplicial}]
		Note that for a simplicial toric variety $X$, we have $\QQ_{X}^{H'} \simeq \IC_{X}^{H}$. This implies that the numbers $a_{\lambda}^{l, j}$ in Theorem \ref{theo:MHM-structure} are zero if the face $\lambda$ is simplicial. Hence, the only non-trivial factors in the structure for $\QQ_{X}^{H'}$ are $\IC_{X}^{H}$ and various Hodge structures supported on the torus fixed point $x_{\sigma}$. Let $i :\{ x_{\sigma} \} \hookrightarrow X$ be the inclusion and notice that $i\sta \QQ_{X}^{H'} \simeq \QQ_{x_{\sigma}}^{H} [n]$. We use the spectral sequence
		$$ E_{2}^{p,q} = \cH^{p} (i\sta \cH^{q}\QQ_{X}^{H'}) \implies \cH^{p+q} i\sta \QQ_{X}^{H'} = \begin{cases} \QQ_{x_{\sigma}}^{H} & p+q = -n \\0 &\text{otherwise.} \end{cases} $$
		For $q < 0$, the Hodge module $\cH^{q} \QQ_{X}^{H'}$ is already supported at the point, and therefore the only non-trivial cohomologies for $\cH^{p}(i\sta \cH^{q}\QQ_{X}^{H'})$ is when $p = 0$ for $q < 0$. Notice that $E_{2}^{0,0} = E_{\infty}^{0, 0}$ and $E_{2}^{-1,0} = E_{\infty}^{-1,0}$ and therefore they are both zero. We have the canonical surjection $\cH^{0}\QQ_{X}^{H'} \to \IC_{X}^{H}$ which is an isomorphism outside $x_{\sigma}$. Consider
		$$ 0 \to \cK \to \cH^{0}\QQ_{X}^{H'} \to \IC_{X}^{H} \to 0$$
		where $\cK$ is a Hodge module supported at $x_{\sigma}$. Since $E_{2}^{-1,0} = E_{2}^{0,0} = 0$, we have the isomorphism
		$$ \cK \simeq \cH^{-1} i\sta \IC_{X}^{H} $$
		and
		$$ \cH^{p} i\sta \cH^{0}\QQ_{X}^{H'} \simeq \cH^{p} i\sta \IC_{X}^{H} \quad \text{for } p \leq -2. $$
		Since the spectral sequence converges to zero except when $p + q = -n$, we see that
		$$ d_{-p}^{p,0} : \cH^{p} i\sta \cH^{0} \QQ_{X}^{H'} \to \cH^{0} i\sta \cH^{p+1} \QQ_{X}^{H'} $$
		has to be an isomorphism for $-n+1 \leq p \leq -2$, and therefore,
		$$ \cH^{-l}\QQ_{X}^{H'} \simeq \cH^{-l-1} i\sta \IC_{X}^{H} $$
		for $1 \leq l \leq n-2$. We use the result of \cite{KV:ICTV} to compute the cohomology of the stalks of $\IC_{X}^{H}$. Let $f_{i}$ be the number of $i$-dimensional faces of $\sigma$, and let
		$$ h_{j} = \sum_{l = j}^{n-1}(-1)^{l-j} {l \choose j} f_{n-1-l}. $$
		The numbers $h_{j}$ encode the Betti numbers of a projective simplicial toric variety, hence these numbers satisfy Hard Lefschetz and Poincaré duality property \cite{Stanley:Intersection-cohomology-toric-varieties}.
        
        For simplicity, we divide the case by the parity of $n$. First, consider the case when $n$ is even. Let $n = 2k$. Then we have
		$$ 1 = h_{0} \leq h_{1} \leq \ldots \leq h_{k-1} = h_{k} \geq h_{k+1} \geq \ldots h_{n-1} = 1.$$
		Moreover, we have
		$$ \cH^{-n+2l} i\sta \IC_{X}^{H} = \QQ_{x_{\sigma}}^{H}(-l)^{h_{l} - h_{l-1}}$$
		for $l < k$ and all the other cohomologies are zero. We use the convention that $h_{-1} = 0$. This shows that
		\begin{align*}
			\cH^{0} \QQ_{X}^{H'} & \simeq \IC_{X}^{H} \\
			\cH^{-n+2l+1} \QQ_{X}^{H'} & \simeq \QQ_{x_{\sigma}}^{H}(-l)^{h_{l} - h_{l-1}} \text{ for } l = 1, \ldots, k-1.
		\end{align*}
		
		Consider the case when $n$ is odd. Let $n = 2k + 1$. Then we have
		$$ 1 = h_{0} \leq h_{1} \leq \ldots \leq h_{k-1} \leq h_{k} \geq h_{k+1} \geq \ldots \geq h_{n-1} = 1.$$
		Here, we have that
		$$ \cH^{-n + 2l} i\sta \IC_{X}^{H} = \QQ_{x_{\sigma}}^{H}(-l)^{h_{l} - h_{l-1}}$$
		for $0 \leq l \leq k$ with the same convention $h_{-1} = 0$. This implies
		\begin{align*}
			\gr_{n}^{W} \cH^{0} \QQ_{X}^{H'} & \simeq \IC_{X}^{H} \\
			\gr_{n-1}^{W} \cH^{0}\QQ_{X}^{H'} & \simeq \QQ_{x_{\sigma}}^{H}(-k)^{h_{k}- h_{k-1}} \\
			\cH^{-n+2l+1} \QQ_{X}^{H'} & \simeq \QQ_{x_{\sigma}}^{H}(-l)^{h_{l} - h_{l-1}} \text{ for } l = 1, \ldots, k-1.
		\end{align*}
		In both cases, we have $\cH^{-n+3} \QQ_{X}^{H'} \simeq \QQ_{x_{\sigma}}^{H}(-1)^{h_{1} - h_{0}} \neq 0$ as soon as $\sigma$ is non-simplicial since $h_{1} - h_{0} = f_{1} - n$.
	\end{proof}

    \subsection{Grothendieck dual of $\duBois_{X}^{p}$} We prove Proposition \ref{prop:Gro-dual-isolated-non-simplicial}.
    \begin{proof}[Proof of Proposition \ref{prop:Gro-dual-isolated-non-simplicial}]
        We insert a ray $\rho$ in the interior of $\sigma$ and consider the corresponding toric morphism $\pi \colon \widetilde{X} \to X$. Let $E$ be the inverse image of the torus fixed point of $X$. We see that $E$ is a projective simplicial toric variety of dimension $n-1$. Note that the singular cohomology of $E$ is pure of Hodge--Tate type by \cite{dCMM-toricmaps}*{Theorem A}, and $\HH^{i}(E, \Ish_{E}^{l}) \simeq H^{n-1-i}(E, \duBois_{E}^{n-1-l})\dual$ by Grothendieck duality. This shows in particular that $\dim \HH^{i}(\Ish_{E}^{j}) = \delta_{ij} h_{i}$ (see \cite{dCMM-toricmaps}*{Corollary C}).
        By Theorem \ref{theo:Lefschetz}, we see that we have the exact sequence
        $$ 0 \to H^{l-1}(\Ish_{\sigma}^{l}) \to \HH^{l-1}(\Ish_{E}^{l-1}) \xrightarrow{c_{1}\dual} \HH^{l}(\Ish_{E}^{l}) \to H^{l}(\Ish_{\sigma}^{l}) \to 0,$$
        and $H^{i}(\Ish_{\sigma}^{l}) = 0$ for $i \neq l-1, l$. By hard Lefschetz, we see that $c_{1}\dual$ is injective if $2(l-1) < n-1$ and surjective if $2l > n-1$. 
        We consider the case when $n = 2k$ is even. Then we have
        \begin{align*}
            \dim H^{l}(\Ish_{\sigma}^{l}) & = h_{l} - h_{l-1} \qquad \text{if } l < k \\
            \dim H^{l-1}(\Ish_{\sigma}^{l}) &= h_{l-1} - h_{l} \qquad \text{if } l > k.
        \end{align*}
        This implies the assertion when $n$ is even. One can deduce the case when $n$ is odd similarly.
    \end{proof}

    \subsection{The graded de Rham complex of the intersection cohomology Hodge module} We compare the Du Bois complex and the graded de Rham complex of the intersection cohomology in this case.
    \begin{prop} \label{prop:isolatednonsimplicial-IDuBois-isomorphism}
        Let $X$ be a toric variety of dimension $n$ with isolated non-simplicial locus. Then for $p \geq n/2$, the natural morphism
        $$ \duBois_{X}^{p}[n-p] \simeq \gr_{-p} \DR \QQ_{X}^{H'} \to \gr_{-p} \DR \IC_{X}^{H}$$
        is an isomorphism. For $p \leq n/2$, the natural morphism
        $$ \gr_{-p} \DR \IC_{X}^{H} \to \gr_{n-p} \DR \DD \QQ_{X}^{H'} \simeq \bfR\SheafHom_{\cO_{X}}(\duBois_{X}^{n-p}, \omega_{X}) [n-p]$$
        is an isomorphism.
    \end{prop}
    \begin{proof}
        It is enough to show the first assertion since $\DD \IC_{X}^{H} \simeq \IC_{X}^{H}(n)$ and the second statement follows from the first by dualizing. By Theorem \ref{theo:cone-over-simplicial}, we see that the only terms appearing in the structure theorem are $\QQ_{x_{\sigma}}^{H}(-j)$ for $j < n/2$. Since $\gr_{-p}\DR \QQ_{x_{\sigma}}^{H} (-j) = 0$ for $j < n/2 \leq p$, the assertion follows.
    \end{proof}
    \begin{rema}
        In \cite{Popa-Park:lefschetz}, they study the `condition $(\ast)_{k}$' requiring isomorphisms between the Du Bois complex and the (shifted) graded de Rham complex of the intersection cohomology up to a certain level. We point out that the range of the isomorphism we have in the proposition above is directly opposed to this condition.
    \end{rema}

    \section{Cones over simple polytopes}\label{section:cone-over-simple-polytope}
    We now discuss the case of a cone over a simple polytope. In fact, the affine toric varieties associated to such cones are precisely those for which all the torus-invariant divisors are simplicial.

    \subsection{Finer control on depth of $\duBois_{X}^{l}$}
    First, we prove Theorem \ref{theo:simple-upperbound-depth}. We show that it is enough to prove the following lemma.
	
	\begin{lemm} \label{lemm:exactness-of-simple-dim-c}
		Let $\sigma$ be a cone, which is simple in dimension $c$. For each $l$, consider the complex
		$$ \Ish_{\sigma}^{l}: V_{0}^{l} \to \bigoplus_{\mu\lind{1} \in \cP_{1}} V_{\mu\lind{1}}^{l} \to \bigoplus_{\mu\lind{2} \in \cP_{2}} V_{\mu\lind{2}}^{l} \to \ldots \to \bigoplus_{\mu\lind{l} \in \cP_{l}} V_{\mu\lind{l}}^{l} $$
		sitting in cohomological degrees 0 to $l$. Then $H^{i}(\Ish_{\sigma}^{l}) = 0$ for $i > c$.
	\end{lemm}
	
	\begin{proof}[Proof of Theorem \ref{theo:simple-upperbound-depth}]
		By Theorem \ref{theo:Gro-dual-of-Ishida-is-duBois}, we see that $\bfR\SheafHom_{\cO_{X}}(\duBois_{X}^{n-l}, \omega_{X}) \simeq \Ish_{X}^{l}$. Fix $u \in \tau_{\circ}\sta$ for some face $\tau \in \cP$. By Lemma \ref{lemm:grade-parts-of-Ishida-complex}, the degree $u$-piece of $\tau_{\circ}\sta$ is described as
		$$ \bigoplus_{j = 0}^{l} \bigwedge^{j} \tau^{\perp} \otimes \Ish_{\tau}^{l-j}.$$
		Note that $\tau$ is also simple in dimension $c$ (if $d_{\tau} \leq c$, there is nothing to prove). Hence, applying Lemma \ref{lemm:exactness-of-simple-dim-c} to $\tau$, we get the required vanishing. Hence $\SheafExt_{\cO_{X}}^{i}(\duBois_{X}^{n-l}, \omega_{X}) = 0$ for all $l$ and all $i > c$. In particular, $\lcdef(X) \leq c-1$ if $c \geq 1$ by Theorem \ref{theo:MP-lcd-intermsof-DuBois}.
	\end{proof}
	
	We now prove Lemma \ref{lemm:exactness-of-simple-dim-c}.
	\begin{proof}
		The strategy of the proof is mainly twofold. First, for $m \geq c$, and $\mu \subset \tau$ such that $\tau\in \cP_{l}$ and $\mu \in \cP_{m}$, we construct one-dimensional subspaces $V_{\mu, \tau}$ of $V_{\mu}^{l}$. This will serve as the correct basis that we work. After that, we use the \textit{lexicographic shelling order} explained in \S\ref{section:shelling} to represent an element in the kernel of the map $(\Ish_{\sigma}^{l})^{m+1} \to (\Ish_{\sigma}^{l})^{m+2}$ by a linear combination of the images of the vectors in $V_{\mu,\tau}$.
		
		We construct the one-dimensional vector spaces $V_{\mu, \tau}$ that satisfy the following properties:
		\begin{enumerate}
			\item For $\mu \in \cP_{m}$, we have $V_{\mu}^{l} \simeq \bigoplus_{\tau \in \cP_{l}^{\supset \mu}} V_{\mu, \tau}$.
			\item Let $\mu \subset \lambda$ such that $\mu \in \cP_{m}$ and $\lambda \in \cP_{m+1}$. Suppose $\mu \subset \tau$. Then $\varphi_{\mu,\lambda} : V_{\mu}^{l} \to V_{\lambda}^{l}$ sends $V_{\mu, \tau}$ isomorphically to $V_{\lambda,\tau}$ if $\lambda \subset \tau$. If $\lambda \not \subset \tau$, then $\varphi_{\mu,\lambda}$ sends $V_{\mu, \tau}$ to zero.
		\end{enumerate}
		We do this inductively on $m$ in a decreasing manner. For $m = l$, we simply put $V_{\tau, \tau} = V_{\tau}^{l}$. For $m = l-1$, pick $\mu \in \cP_{l-1}$. Then we observe that
		$$ \mu^{\perp} \to \bigoplus_{\tau \in \cP_{l}^{\supseteq \mu}} \RR_{\tau} $$
		is an isomorphism. Here, $\RR_{\tau} = V_{\tau}^{l} = \RR$ is a one-dimensional vector space. Indeed, one can easily check that this morphism is injective since $\mu^{\perp} = \cap_{\tau \in \cP_{l}^{\supset \mu}} \tau^{\perp}$, and the dimension matches up since $|\cP_{l}^{\supset \mu}| = n-l+1$ by the simpleness assumption. We simply take $V_{\mu, \tau}$ as the pre-image of $\RR_{\tau}$. Note that for $V_{\mu, \tau}$, the natural evaluation map to all the other $\RR_{\tau'}$ are zero, and $V_{\mu, \tau} \xrightarrow{\simeq} \RR_{\tau}$. We suppose that $V_{\mu, \tau}$ is constructed for the faces with dimension $m + 1, \ldots, l$ satisfying the properties above, and also assume $m + 2\leq l$.
		
		\noindent
		By Lemma \ref{lemm:exact-sequence-from-simplicial}, the following sequence is exact:
		$$ V_{\mu}^{l} \to \bigoplus_{\lambda\lind{m+1} \in \cP_{m+1}^{\supset \mu}} V_{\lambda\lind{m+1}}^{l} \to \ldots \to \bigoplus_{\lambda\lind{l} \in \cP_{l}^{\supset \mu}} V_{\lambda\lind{l}}^{l} .$$
		Pick $\tau \in \cP_{l}^{\supset \mu}$. Note that there are exactly $l-m$ elements $\lambda_1,\ldots, \lambda_{l-m}$ of $\cP_{m+1}^{\supset \mu} \cap \cP_{m+1}^{\subset \tau}$. We see that $\cP_{m+2}^{\supset \mu} \cap \cP_{m+2}^{\subset \tau} = \{ \lambda_{ij}\}_{1 \leq i < j \leq l-m}$, where $\lambda_{ij}$ is the unique face both containing $\lambda_{i}$ and $\lambda_{j}$. Pick a non-zero vector $v_{1} \in V_{\lambda_1, \tau}^{l}$. Then for $j = 2, \ldots, l-m$, we have unique non-zero vectors $v_{j} \in V_{\lambda_{j}, \tau}$ such that
		$$ \varphi_{\lambda_{1}, \lambda_{1j}} (v_{1}) + \varphi_{\lambda_{j}, \lambda_{1j}}(v_{j}) = 0,$$
		since $V_{\lambda_{1},\tau}$ and $V_{\lambda_{j}, \tau}$ both map isomorphically to $V_{\lambda_{1j}, \tau}$. View $\sum_{i=1}^{l-m} v_{i}$ as an element in $\bigoplus_{\lambda\lind{m+1} \in \cP_{m+1}^{\supset \mu}} V_{\lambda\lind{m+1}}^{l}$ in the complex above. For $2 \leq i < j \leq l-m$, let $\lambda_{1ij}\in \cP_{m+2}$ be the unique face containing $\lambda_{1}, \lambda_{i}$, and $\lambda_{j}$. Then,
		\begin{align*}
			& \varphi_{\lambda_{ij}, \lambda_{1ij}}(\varphi_{\lambda_{i}, \lambda_{ij}}(v_{i}) + \varphi_{\lambda_{j}, \lambda_{ij}}(v_{j})) \\ &+ \varphi_{\lambda_{1i},\lambda_{1ij}}(\varphi_{\lambda_{1},\lambda_{1i}}(v_1) + \varphi_{\lambda_{i}, \lambda_{1i}}(v_{i})) \\ &+ \varphi_{\lambda_{1j},\lambda_{1ij}}(\varphi_{\lambda_{1},\lambda_{1j}}(v_1) + \varphi_{\lambda_{j}, \lambda_{1j}}(v_{j})) = 0 \in V_{\lambda_{1ij}}^{l}.
		\end{align*}
		Since the second and the third term are zeros, we know that
		$$ \varphi_{\lambda_{i}, \lambda_{ij}}(v_{i}) + \varphi_{\lambda_{j}, \lambda_{ij}}(v_{j}) = 0$$
		since $\varphi_{\lambda_{ij}, \lambda_{1ij}} : V_{\lambda_{ij}, \tau} \to V_{\lambda_{1ij}, \tau}$ is an isomorphism. This shows that $\sum_{j=1}^{l-m} v_{j}$ is in the kernel of the map $\bigoplus_{\lambda\lind{m+1} \in \cP_{m+1}^{\supset \mu}} V_{\lambda\lind{m+1}}^{l} \to \bigoplus_{\lambda\lind{m+2} \in \cP_{m+2}^{\supset \mu}} V_{\lambda\lind{m+2}}^{l} $. Hence, there exists a non-zero element in $v_{\mu, \tau} \in V_{\mu}^{l}$ mapping to this element. We define $V_{\mu, \tau}$ to be the vector space spanned by $v_{\mu, \tau}$. It is clear that this subspace satisfies the second property. Also, consider the expression
		$$ \bigoplus_{\lambda \in \cP_{m+1}^{\supset \mu}} V_{\lambda}^{l} = \bigoplus_{\tau \in \cP_{l}^{\supset \mu}} \bigoplus_{\lambda \in \cP_{m+1}^{\subset \tau} \cap \cP_{m+1}^{\supset \mu}} V_{\lambda, \tau}. $$
		We see that each $V_{\mu, \tau}$ is mapped to $\bigoplus_{\lambda \in \cP_{m+1}^{\subset \tau} \cap \cP_{m+1}^{\supset \mu}} V_{\lambda, \tau}$. This shows that $\{V_{\mu, \tau}\}_{\tau \in \cP_{l}^{\supset \mu}}$ are linearly independent. Note that $\dim V_{\mu}^{l} = {n-m \choose l-m} = |\cP_{l}^{\supset \mu}|$. Therefore, we see that $V_{\mu, \tau}$ should also satisfy the first property as well. This finishes the construction of the spaces $V_{\mu, \tau}$.
		
		The second step is to show the exactness of $\Ish_{\sigma}^{l}$ on cohomological degree $m+1 > c$. Consider an element in the kernel of the map
		$$ \bigoplus_{\lambda \in \cP_{m+1}} V_{\lambda}^{l} \to \bigoplus_{\nu \in \cP_{m+2}} V_{\nu}^{l}. $$
		Using $V_{\lambda} = \bigoplus_{\tau \in \cP_{l}^{\lambda \subseteq}} V_{\lambda, \tau}$ and similarly for $V_{\nu}$, we write an element in the kernel as $ \sum_{\tau}\sum_{\lambda \subseteq \tau} v_{\lambda, \tau}$, where $v_{\lambda, \tau} \in V_{\lambda, \tau}$. Note that for fixed $\tau$, the summand $\sum_{\lambda \subset \tau} v_{\lambda, \tau}$ maps to $\bigoplus_{\nu \in \cP_{m+2}^{\subseteq \tau}} V_{\nu, \tau}$, hence lies in the kernel as well. Therefore, we work with a fixed $\tau$.
		
		We consider a lexicographic shelling order $\prec_{\mathrm{lex}}$ on the set
		$$ \Delta = \Big\{ (\tau,\rho\lind{l-1},\ldots, \rho\lind{m+1}) : \tau \supset \rho\lind{l-1}\supset \ldots \supset \rho\lind{m+1}, d_{\rho\lind{i}} = i \Big\},$$
		which is described in \S\ref{section:shelling}. We consider the subset
		$$ \cS = \{(\tau, \rho\lind{l-1},\ldots, \rho\lind{m+1}) \in \Delta : v_{\rho\lind{m+1}, \tau} \neq 0 \}$$
		and consider the maximal element $(\tau, \rho\lind{l-1},\ldots, \rho\lind{m+1})$ of $\cS$ (if $\cS$ is empty, we have nothing to show). First, we claim that $\rho\lind{m+1}$ cannot be the minimal element with respect to $\prec_{\tau,\rho\lind{l-1},\ldots, \rho\lind{m+2}}$. This is because $v_{\rho\lind{m+1}, \tau} \neq 0$, hence there should be a facet $\widetilde{\rho}$ of $\rho\lind{m+2}$ such that $v_{\widetilde{\rho}, \tau}\neq 0$ as well. This violates the maximality. Hence, there exists some facet $\widetilde{\rho}\lind{m+1}$ of $\rho\lind{m+2}$ such that $\rho = \rho\lind{m+1} \cap \widetilde{\rho}\lind{m+1} \in \cP_{m}$ and $\widetilde{\rho}\lind{m+1} \prec_{\tau,\rho\lind{l-1},\ldots, \rho\lind{m+2}} \rho\lind{m+1}$. The goal is to show that
        $$(\tau, \rho\lind{l-1}',\ldots, \rho\lind{m+1}') \prec_{\mathrm{lex}} (\tau, \rho\lind{l-1},\ldots, \rho\lind{m+1})$$
        for all chains satisfying $(\tau, \rho\lind{l-1}',\ldots, \rho\lind{m+1}') \in \Delta$ such that $\rho\lind{m+1}' \supset \rho$. If this is the case, we can pick an element $v \in V_{\rho, \tau}$ that maps to $v_{\rho\lind{m+1}, \tau}$ via $\varphi_{\rho, \rho\lind{m+1}}$. After subtracting the image of $v$ from $\sum_{\lambda \subset \tau}$, we see that the maximum of the set $\cS$ descreases, since all the other chains from $\tau$ to $\rho$ are strictly smaller than $(\tau, \rho\lind{l-1}, \ldots, \rho\lind{m+1})$. This shows the exactness. Therefore, it remains to show this combinatorial fact.
        
        We consider the poset $\cP^{\rho \subseteq} \cap \cP^{\subseteq \tau}$, which is isomorphic to the poset of all subsets of $\{1, \ldots, l-m\}$ due to the simpleness assumption. The face $\rho$ corresponds to $\emptyset$ and the chain $\rho \subset \rho\lind{m+1} \subset \ldots \subset \tau$ corresponds to the chain
        $$ \emptyset \subset \{ 1\} \subset \{ 1, 2\} \subset \ldots \subset \{ 1, 2, \ldots, l -m \}.$$
        It is easy to see that for $k = m+1,\ldots, l-1$, there is a unique element $\widetilde{\rho}\lind{k}$ in $\cP_{k}$ that is contained in $\rho\lind{k+1}$, but does not contain $\rho\lind{m+1}$, with the convention that $\rho\lind{l} = \tau$. Indeed, using the correspondence above, $\widetilde{\rho}\lind{k}$ corresponds to subset $\{2, \ldots, k-m+1\}$. Also, it is easy to see that $\widetilde{\rho}\lind{m+1}$ is the one we mentioned before, and
		$$\rho \subseteq \widetilde{\rho}\lind{m+1} \subseteq \ldots \subseteq \widetilde{\rho}\lind{l-1}.$$
		Since $\widetilde{\rho}\lind{m+1} \prec_{\tau,\ldots,\rho\lind{m+2}} \rho\lind{m+1}$ and $\rho\lind{m+2} \cap \widetilde{\rho}\lind{m+2} = \widetilde{\rho}\lind{m+1}$, we have $\widetilde{\rho}\lind{m+2} \prec_{\tau,\ldots, \rho\lind{m+3}} \rho\lind{m+2}$. Inductively, we get
		$$ \widetilde{\rho}\lind{k} \prec_{\tau, \ldots, \rho\lind{k+1}} \rho\lind{k} $$
		for all $k = m+1, \ldots, l-1$.
		
		We consider $(\tau, \rho\lind{l-1}',\ldots, \rho\lind{m+1}') \in \Delta$ such that $\rho\lind{m+1}' \supset \rho$. Suppose that $k$ is the largest number such that $\rho\lind{k}' \neq \rho\lind{k}$. Since $\rho\lind{k}'$ contains $\rho$ and is contained in $\rho\lind{k+1}$, we have either $\rho\lind{k}' = \widetilde{\rho}\lind{k}$, or $\rho\lind{k}' \supset \rho\lind{m+1}$. For the first case, we see that $\widetilde{\rho}\lind{k} \prec_{\tau,\ldots, \rho\lind{k+1}} \rho\lind{k}$ by the previous observation. For the second case, we can pick another chain
		$$ (\tau, \rho\lind{l-1}' ,\ldots, \rho\lind{k}', \rho\lind{k-1}'', \ldots \rho\lind{m+1}'') \in \Delta$$
		such that $\rho\lind{m+1}'' = \rho\lind{m+1}$. Then this chain is in $\cS$, and therefore, we have $\rho\lind{k}' \prec_{\tau,\ldots, \rho\lind{k+1}} \rho\lind{k}$ as well. This concludes the proof.
	\end{proof}

    \subsection{Structure of $\QQ_{X}^{H'}$}
    We now give the proof of Theorem \ref{theo:cone-over-simple}.

    \begin{proof}[Proof of Theorem \ref{theo:cone-over-simple}]
         Let $\sigma$ be a simple $n$-dimensional cone. We notice that all the faces of $\sigma$ are simple as well. Also, all the proper torus-invariant closed subvarieties $S_{\tau}$ are simplicial, hence $\IC_{S_{\tau}}^{H} \simeq \QQ_{S_{\tau}}^{H'}$. By induction, we assume that the formula holds for $a_{\tau}^{0, j}$ for all faces $\tau \subsetneq \sigma$. Note that $\lcdef(X) = 0$ and hence $\QQ_{X}^{H'}$ is a mixed Hodge module by Theorem \ref{theo:lcdef-of-simple}. Also, we immediately see that $a_{\sigma}^{0, j} = 0$ for $j \geq n/2$ by Theorem \ref{theo:MHM-structure}. Hence, it is enough to only consider $a_{\sigma}^{0, j}$ for $j < n/2$.
         
    \noindent
    \textbf{Step 1.} Relation between the face numbers.
    
    Take a hyperplane $H \subset N_{\RR}$ intersecting $\sigma$ so that $H \cap \sigma$ is a simple polytope of dimension $n-1$. Then the polynomial
    $$ (q^{2}-1)^{n-1} + f_{n-1} (q^{2}-1)^{n-2} + \ldots + f_{2}(q^{2}-1) + f_{1}$$
    is the Poincar\'{e} polynomial for the toric variety corresponding to the polar polytope of $H \cap \sigma$. Consider
    $$\widetilde{h}_{t} = \sum_{l=0}^{t} f_{n-l} {n-1-l \choose t-l} (-1)^{t-l}$$
    which is the coefficient of $q^{2(n-1-t)}$ of the polynomial above. Then we have
    $$ \widetilde{h}_{0} \leq \widetilde{h}_{1} \leq \ldots \leq \widetilde{h}_{\lround{\frac{n-1}{2}}}$$
    and $\widetilde{h}_{t} = \widetilde{h}_{n-1-t}$.

    We note that
    \begin{align*}
        \widetilde{h}_{t} - \widetilde{h}_{t-1} & = f_{n-t} +\sum_{l=0}^{t-1} \left[(-1)^{t-l} {n-1+l \choose t - l} - (-1)^{t-l-1} {n-1-l \choose t-1- l} \right] f_{n-l} \\
        & = f_{n-t} + \sum_{l=0}^{t-1} (-1)^{t-l} \left[ {n-1-l \choose t-l} + {n-1-l \choose t - 1- l} \right] f_{n-l} \\
        & = \sum_{l=0}^{t} (-1)^{t-l} {n-l \choose t - l} f_{n-l}.
    \end{align*}
    This is exactly the right-hand side for the formula for $a_{\sigma}^{0, t}$ in this Theorem. We sometimes denote $\widetilde{h}_{t}$ by $\widetilde{h}_{t}^{\sigma}$ if we want to emphasize the cone.

    \noindent
    \textbf{Step 2.} Structure of $\QQ_{X}^{H'}$ and $\DD \QQ_{X}^{H'}$, and the Ishida complex.
    
    From Theorem \ref{theo:MHM-structure} and the fact that $\QQ_{X}^{H'}$ is a genuine mixed Hodge module, and the fact that $\IC_{S_{\tau}}^{H} \simeq \QQ_{S_{\tau}}^{H'}$ for all non-zero faces $\tau \in \cP$, we have
    \begin{align*}
        \gr_{n}^{W} \QQ_{X}^{H'} & \simeq \IC_{X}^{H} \\
        \gr_{n-t} \QQ_{X}^{H'} & \simeq \bigoplus_{j \geq 1} \bigoplus_{\lambda \in \cP_{t + 2j}} \QQ_{S_{\lambda}}^{H'}(-j)^{a_{\lambda}^{0, j}} \quad \text{for } 1\leq t \leq n-2.
    \end{align*}
    Dually, we get
    \begin{align*}
        \gr_{-n}^{W} \DD \QQ_{X}^{H'} & \simeq \IC_{X}^{H}(n) \\
        \gr_{-r}^{W} \DD \QQ_{X}^{H'} & \simeq \bigoplus_{j \geq 1} \bigoplus_{\lambda \in \cP_{n-r+ 2j}} \QQ_{S_{\lambda}}^{H'} (r-j)^{a_{\lambda}^{0,j}} \quad \text{for } 2\leq r \leq n-1,
    \end{align*}
    putting $r = n-t$. For each $1 \leq j < n/2$, we consider the spectral sequence
    $$ E_{1}^{r,s} = \cH^{r+s} \left( \gr_{j} \DR \gr_{-r}^{W} \DD \QQ_{X}^{H'}\right)_{0} \implies \SheafExt_{\cO_{X}}^{j+r+s}(\duBois_{X}^{j}, \omega_{X})_{0}. $$
    We note that the right hand side is computed by the cohomology of $\Ish_{\sigma}^{n-j}$ and by Lemma \ref{lemm:exactness-of-simple-dim-c}, we see that $H^{1}(\Ish_{\sigma}^{n-j})$ is the only non-vanishing cohomology of the Ishida complex.

    We also note that
    $$ \cH^{r+s} \left( \gr_{j} \DR \QQ_{S_{\lambda}}^{H'} (n-t-l)\right)_{0}$$
    is non-zero only when $j = n-t-l$ and $r+s = - \dim S_{\lambda}$. This follows by the fact that the graded de Rham complexes of $\QQ_{S_{\lambda}}^{H'}$ are the (shifted) reflexive differentials and by Lemma \ref{lemm:dimension-graded-of-reflex-diff}.

    In order for the term $E_{1}^{r,s}$ to be non-zero, we need $r \geq j+1$ and $n - r + 2(r-j) \leq n$. Hence $j+1 \leq r \leq 2j$. For these $r$, the non-zero factor for $E_{1}^{r,s}$ comes from $l = r-j$ and hence
    $$ \dim_{\CC} E_{1}^{r,s} = \dim_{\CC} \cH^{r+s}\left( \gr_{j} \DR \bigoplus_{\lambda \in \cP_{n+r-2j}} \QQ_{S_{\lambda}}^{H'}(j)^{a_{\lambda}^{0, r-j}} \right)_{0} = \begin{cases} \sum_{\lambda \in \cP_{n+r-2j}} a_{\lambda}^{0, r-j} & \text{if } s = -2j \\ 0 & \text{otherwise.} \end{cases}$$
    Therefore, the spectral sequence degenerates at $E_{2}$ and this page is computed after taking the cohomology of
    $$ E_{1}^{j+1, -2j} \xrightarrow{d_{1}} E_{1}^{j+2, -2j} \to \ldots \xrightarrow{d_{1}} E_{1}^{2j, -2j}.$$
    Note that $\dim E_{1}^{2j, -2j} = a_{\sigma}^{0, j}$. Since $H^{1}(\Ish_{\sigma}^{n-j})$ is the only non-vanishing cohomology of the Ishida complex, we know that $E_{2}^{j+1, -2j}$ is the only non-zero term in the spectral sequence, and has rank $\dim H^{1} \Ish_{\sigma}^{n-j}$. By calculating the Euler characteristic of the Ishida complex, we see that
    $$ \dim H^{1} \Ish_{\sigma}^{n-j} = - {n \choose j} + f_{1} {n-1 \choose j} + f_{2} {n-2 \choose j} - \ldots + (-1)^{n-j-1} f_{n-j}{j \choose j}.$$
    We therefore have the equality
    \begin{align*}
        & \sum_{\lambda \in \cP_{n-j+1}} \widetilde{h}_{1}^{\lambda} - \sum_{\lambda \in \cP_{n-j+2}} \widetilde{h}_{2}^{\lambda} + \ldots + (-1)^{j-2} \sum_{\lambda \in \cP_{n-1}} \widetilde{h}_{j-1}^{\lambda} + (-1)^{j-1} a_{\sigma}^{0, j} \\
        & =  - {n \choose j} + f_{1} {n-1 \choose j} + f_{2} {n-2 \choose j} - \ldots + (-1)^{n-j-1} f_{n-j}{j \choose j},
    \end{align*}
    using induction and the fact that all faces of $\sigma$ are simple. The rest of the proof boils down to actually showing that the equality $a_{\sigma}^{0, j} = \widetilde{h}_{j}^{\sigma}$ holds.

    \noindent
    \textbf{Step 3.} Proof of the equality
    
    We will show the equality by showing that
    $$ a_{\sigma}^{0, j} = \sum_{l=0}^{j} (-1)^{j-1} {n-1-l \choose j-l} (\widetilde{h}_{n-1-l} - \widetilde{h}_{l}) + \widetilde{h}_{j} - \widetilde{h}_{j-1}.$$
    We remark that the first term of the right hand side is zero. We divide these into three terms. Hence, it is enough to verify the equality
    \begin{align} \label{eqn:master-equality-simple}
        &\underbrace{\sum_{l=0}^{j} {n-1-l \choose j-l} \widetilde{h}_{n-1-l}}_{I} - \underbrace{\sum_{l=0}^{j} {n-1-l \choose j-l} \widetilde{h}_{l}}_{II} + \underbrace{ (-1)^{j-1} (\widetilde{h}_{j} - \widetilde{h}_{j-1})}_{III} \\
        & =\underbrace{ \left(\sum_{l=1}^{j-1} (-1)^{l} \sum_{\lambda \in \cP_{n-j+l}} \widetilde{h}_{l}^{\lambda} \right)}_{IV} +\left( \sum_{l=1}^{n-j} (-1)^{l-1} f_{l} {n-l \choose j}\right) - {n\choose j}. \nonumber
    \end{align}

    \noindent
    \textbf{Term IV}. We first simplify Term IV in equation (\ref{eqn:master-equality-simple}). We first note that
    \begin{align*}
         \sum_{\lambda \in \cP_{n-j+l}} \widetilde{h}_{l}^{\lambda} 
        & = \sum_{\lambda \in \cP_{n-j+l}} \sum_{i=0}^{l} (-1)^{l} f_{n-j+i}^{\lambda} {n - j + i \choose i} \\
        & = \sum_{i=0}^{l} (-1)^{i} f_{n-j+i}^{\sigma} {j -i \choose l-i} {n-j+i \choose i}.
    \end{align*}
    Here, $f_{n-j+i}^{\lambda}$ denotes the number of $n-j+i$-dimensional faces of $\lambda$. The last equality follows from the fact that for each $n-j+i$-dimensional face, there are exactly ${j-i \choose l-i}$ number of $n-j+l$-dimensional faces containing it, since $\sigma$ is simple. Then the equality is a standard double-counting argument.
    
    Then Term IV simplifies as
    \begin{align*}
         \sum_{l=1}^{j-1} \left( \sum_{\lambda \in \cP_{n-j+l}} \widetilde{h}_{l}^{\lambda} \right) (-1)^{l}
        & = \sum_{l=1}^{j-1} \sum_{i=0}^{l} (-1)^{i+l} f_{n-j+i} {j-i \choose l - i} {n-j+i \choose i} \\
        & =\underbrace{ \sum_{l=1}^{j-1} (-1)^{l} f_{n-j} {j \choose l}}_{i=0 \text{ term}} + \sum_{i=1}^{j-1} f_{n-j+i} {n-j+i \choose i} \underbrace{\sum_{l=i}^{j-1} (-1)^{l-i} {j-i \choose l-i}}_{= (-1)^{j-i-1}} \\
        & = - (1 + (-1)^{j}) f_{n-j} + \sum_{i=1}^{j-1} (-1)^{j-i-1} f_{n-j+i} {n-j+i \choose i} \\
        & = -f_{n-j} + \sum_{i=0}^{j-1} (-1)^{j-1-i} f_{n-j+i} {n-j+i \choose i} \\
        & = - f_{n-j} + \sum_{l=1}^{j} (-1)^{l-1} f_{n-l} {n-l \choose j-l}.
    \end{align*}

    \noindent
    \textbf{Term I}. We give a small lemma before considering the term I.
    \begin{lemm}
    For positive integers $n, j$ and non-negative integer $i \in [0, n-1]$ such that $j \leq n-1$, we have
    $$ \sum_{l=0}^{\min (j, n-1-i)} (-1)^{l} {n-1-l \choose j -l} \cdot {n-1-i \choose l} = {i \choose j}.$$
    \end{lemm}
    \begin{proof}
        This equality can be checked by computing the number of size $j$ subsets of $\{1, \ldots n-1\}$ contained in $\{ 1, \ldots, i\}$ using the inclusion-exclusion principle.
    \end{proof}
    Then term I simplifies as follows:
    
    \begin{align*}
         \sum_{l=0}^{j} {n-1-l \choose j-l} \widetilde{h}_{n-1-l} 
        & = \sum_{l=0}^{j} \sum_{i=0}^{n-1-l} f_{n-i} {n-1-l \choose j-l} {n-1-i \choose l}(-1)^{n-1-l-i} \\
        & = \sum_{i=0}^{n-1} f_{n-i} (-1)^{n-1-i} \sum_{l=0}^{\min(j, n-1-i)} (-1)^{l} {n-1-l \choose j-l} \cdot {n-1-i \choose l} \\
        & = \sum_{i=0}^{n-1} f_{n-i} (-1)^{n-1-i} {i \choose j} \\
        & = \sum_{l=1}^{n-j} f_{l} {n-l \choose j} (-1)^{l-1} \quad \text{by putting $l = n-i$}.
    \end{align*}

    \noindent
    \textbf{Term II}. We consider Term II.
    \begin{align*}
        \sum_{l=0}^{j} {n-1-l \choose j-l} \widetilde{h}_{l} 
        & = \sum_{l=0}^{j} \sum_{i=0}^{l} (-1)^{l-i} {n-1-i \choose l-i}{n-1-l \choose j-l} f_{n-i} \\
        & = \sum_{i=0}^{j} f_{n-i} \sum_{l=i}^{j} (-1)^{l-i} {n-1-i \choose l-i}{n-1-l \choose j-l} \\
        & = \sum_{i=0}^{j} f_{n-i} \sum_{j' = 0}^{j-i} (-1)^{j'} {n-1-i \choose j'} {n-1-i- j' \choose j-i-j'} \quad \text{ by putting } j' = l-i \\
        & = \sum_{i=0}^{l} f_{n-i} \sum_{l = 0}^{j-i} (-1)^{l} \frac{(n-1-i)! \cdot (n-1-i-l)!}{l! \cdot (n-1-i-l)! \cdot (j-i-l)! \cdot(n-j-1)!} \\
        & = \sum_{i=0}^{j} f_{n-i} \sum_{l=0}^{j-i} (-1)^{l} \frac{(n-1-i)!}{(j-i)! \cdot (n-j-1)!} \frac{(j-i)!}{l! (j-i-l)!} \\
        & = \sum_{i=0}^{j} f_{n-i} {n-1-i \choose j-i} \sum_{l=0}^{j-i} (-1)^{l} {j-i \choose l} = f_{n-j}.
    \end{align*}

    \noindent
    \textbf{Term III}. For the last term, we have
    \begin{align*}
         (-1)^{j-1} (\widetilde{h}_{j} - \widetilde{h}_{j-1}) 
        & = \left(\sum_{l=0}^{j} f_{n-l} {n-1-l \choose j-l} (-1)^{l-1}\right) - \left( \sum_{l=0}^{j-1} f_{n-l} {n-1-l \choose j-1-l} (-1)^{l}\right) \\
        & = (-1)^{j-1} f_{n-j} + \sum_{l=0}^{j-1} f_{n-l} (-1)^{l-1} \left( {n-1-l \choose j-l} + {n-1-l \choose j-1-l} \right) \\
        & = (-1)^{j-1} f_{n-j} + \sum_{l=0}^{j-1} f_{n-l} (-1)^{l-1} { n-l \choose j-l} = \sum_{l=0}^{j} f_{n-l}(-1)^{l-1} {n-l \choose j-l}.
    \end{align*}

    Now, we wrap up everything. We see that the left-hand side of Equation \ref{eqn:master-equality-simple} is
    $$ -f_{n-j} + \sum_{l=1}^{n-j} f_{l} {n-l \choose j} (-1)^{l-1} + \sum_{l=0}^{j} f_{n-l} (-1)^{l-1} {n-l \choose j-l}. $$

    We also know that the right-hand side of the Equation \ref{eqn:master-equality-simple} is
    $$ -f_{n-j} + \sum_{l=1}^{j} (-1)^{l-1} f_{n-l} {n-l \choose j-l} + \sum_{l=1}^{n-j} (-1)^{l-1} f_{l} {n-l \choose j} - {n \choose j}. $$
    We see that the two equations match up. This concludes the proof of the proposition. 
    \end{proof}

    \subsection{Singular cohomology of toric varieties corresponding to simple polytopes} \label{subsec:sing-coho-simple}
    Let $P \subset N_{\RR}$ be a lattice polytope containing the origin in the interior. Then one can construct a fan whose faces are $\RR_{\geq 0}$-spans of faces of $P$, hence we can associate a proper toric variety $X_{P}$. One can consider the piecewise linear function $\psi$ on $N_{\RR}$ which is linear on each face, by assigning values $1$ on each vertex and $0$ at the origin. This corresponds to an ample $\QQ$-divisor $D_{P}$ on $X_{P}$, and hence $X_{P}$ is a projective toric variety. If $P$ is a simplicial polytope, the singular cohomology of $X_{P}$ is easy to compute, coming from the fact that the Hodge structures are pure of Hodge--Tate type, and we have an explicit description of the Hodge--Deligne polynomial (see \cite{dCMM-toricmaps}*{Corollary C}). For a general polytope, the purity of the Hodge structure breaks, which makes the computation of the singular cohomology hard. However, if $P$ is a simple polytope, the singular cohomology of $X_P$ has a nice formula as stated in Theorem \ref{theo:sing-coho-simple}. We now give  a proof of the same.

    \begin{proof}[Proof of Theorem \ref{theo:sing-coho-simple}]
        We denote $X_{P}$ by $X$. By Grothendieck duality, we have
        $$ h^{p,q}(X) = \dim \HH^{n-q} (\Ish_{X}^{n-p}).$$
        Hence, we prove the statements in terms of the hypercohomology of the Ishida complexes. Consider the graph of the piecewise linear function $\psi$ in $N_{\RR} \oplus \RR \cdot e_{n+1}$ corresponding to the ample divisor $D_{P}$. Let $\varsigma = \{ x + t e_{n+1} : t \geq \psi(x), x \in N_{\RR} \}$. Since $D_{P}$ is ample, $\varsigma$ is a strictly convex rational polyhedral cone of dimension $n+1$. Note that $\varsigma$ is a cone over a simple polytope; indeed, if we cut by the hyperplane $N_{\RR} \times \{1\}$, we recover $P$. Therefore, $H^{j}(\Ish_{\varsigma}^{i}) = 0$ for $j > 1$. Also, it is easy to directly see that $H^{0}(\Ish_{\varsigma}^{i}) =0$ for $i > 0$. Hence $\Ish_{\varsigma}^{i}$ is concentrated in cohomological degree 1 unless $i = 0$. Now, we use Theorem \ref{theo:Lefschetz} to the following situation: we consider the affine toric variety $X_{\varsigma}$ and $\pi \colon\widetilde{X}_{\varsigma} \to X_{\varsigma}$ the toric morphism obtained by inserting the ray spanned by $e_{n+1}$ in $\varsigma$. Note that the inverse image of the torus fixed point $x_{\varsigma}$ is isomorphic to $X$, and we can relate the singular cohomology of $X$ and the cohomologies of $\Ish_{\varsigma}^{i}$. 
        
        First, it is easy to see that
        $$ \dim \HH^{j}(\Ish_{X}^{0}) = \begin{cases}
            1 & j = 0 \\ 0 & j \neq 0
        \end{cases}$$
        and
        $$ \dim \HH^{j}(\Ish_{X}^{n}) = \dim H^{j}(X, \omega_{X}) \simeq \dim H^{n-j}(X, \cO_{X}) = \begin{cases}
            1 & j = n \\ 0 & j \neq n.
        \end{cases}$$
        Next, we show that $\HH^{q'}(\Ish_{X}^{p'}) = 0$ unless $q' = 1$ or $p' = q'$ by descending induction on $p'$. For this, we use the long exact sequence coming from Theorem \ref{theo:Lefschetz}:
        $$ \ldots \to  \HH^{i-1}(X, \Ish_{X}^{p'-1}) \xrightarrow{c_{1}\dual} \HH^{i}(X, \Ish_{X}^{p'}) \to H^{i}(\Ish_{\varsigma}^{p'}) \to \HH^{i}(X, \Ish_{X}^{p'-1}) \xrightarrow{c_{1}\dual} \ldots. $$
        If $p'>2$, then Theorem \ref{theo:simple-upperbound-depth} implies that $H^{i}(\Ish_{\varsigma}^{p'}) = 0$ for $i \geq 2$. And so, we have
        $$ \HH^{p'-1}(\Ish_{X}^{p'-1}) \simeq \HH^{p'} (\Ish_{X}^{p'}) $$
        and
        $$0 \to \HH^{1}(\Ish_{X}^{p'}) \to H^{1}(\Ish_{\varsigma}^{p'}) \to \HH^{1}(\Ish_{X}^{p'-1}) \to 0,$$
        and all the other terms in the long exact sequence are zero. If $p' = 2$, we have
        $$ 0 \to \HH^{1}(\Ish_{X}^{2}) \to H^{1}(\Ish_{\varsigma}^{1}) \to \HH^{1}(\Ish_{X}^{1}) \to \HH^{2}(\Ish_{X}^{2}) \to 0$$
        and all the other terms in the long exact sequence are zero. This shows the main assertion, except the values of $h^{p,q}$ for $q = n-1$.
    
        We have the following description of the Hodge--Deligne polynomial of $X$:
        $$ E(X):=\sum_{i,r,s} (-1)^{i} h^{r,s} (\gr_{r+s}^{W} H^{i}(X)) u^{r} v^{s} = \sum_{j=0}^{n} f_{j-1}(uv-1)^{n-j}.$$
        We put $f_{-1} = 1$ for convention. Indeed, one can use the additivity of the Hodge--Deligne polynomial and the fact that the Hodge--Deligne polynomial of $\CC^{\times}$ is $uv - 1$ to get the desired equality. Also, note that the Hodge structure of $H^{i}(X)$ is mixed of Hodge--Tate type by Corollary \ref{coro:sing-coho-HT-type}. This shows that for $1 \leq p < n-1$, we have
        $$ 1 + h^{p, n-1}(X) (-1)^{n-1+p} = \mathrm{coeff}_{u^{p}v^{p}} E(X) = \sum_{j=0}^{n} f_{j-1} (-1)^{n-j-p} {n-j \choose p}, $$
        and
        $$ h^{n-1, n-1}(X) = \mathrm{coeff}_{u^{n-1}v^{n-1}} E(X) =f_{0}-n.$$
        This shows the assertion.
    \end{proof}

    \begin{exam} \label{exam:sing_coh_binomial_hypersurface}(Singular cohomology of the binomial hypersurface $S = \{x_{0}\cdots x_{n} = y_{0}\cdots y_{n}\} \subset \PP^{2n+1}$). Here, we compute that singular cohomology of the binomial hypersurface given by the equation above. Consider $X = \Spec \CC[x_{0},\ldots, x_{n}, y_{0}\ldots, y_{n}]/(\prod_{i=0}^{n} x_{i} - \prod_{i=0}^{n} y_{i})$. Note that this ring is isomorphic to 
    $$\CC[e_{1},e_{2}e_{3},\ldots, e_{2n}e_{2n+1}, e_{2n+1}, e_{2n}e_{2n-1},\ldots, e_{2}e_{1}]\subset \CC[e_{1},\ldots, e_{2n+1}]$$
    by sending
    \begin{align*}
        & x_{0} \mapsto e_{1}, x_{1}\mapsto e_{2}e_{3},\ldots, x_{n} \mapsto e_{2n} e_{2n+1} \\
        & y_{0} \mapsto e_{2n+1}, y_{1} \mapsto e_{2n}e_{2n-1},\ldots, y_{n} \mapsto e_{2}e_{1}.
    \end{align*}
    Note that $X = X_{\sigma}$ is an affine toric variety, where $\sigma$ is a $2n+1$-dimensional full-dimensional cone whose dual cone $\sigma\dual$ is spanned by the vectors
    \begin{align*}
        & v_{0} = e_{1}, \quad v_{1} = e_{2} + e_{3},\quad v_2 = e_4 + e_5, \quad \ldots \quad v_{n} = e_{2n} + e_{2n+1} \\
        & u_{0} = e_{2n+1}, \quad u_{1} = e_{2n} + e_{2n-1}, \quad u_2 = e_{2n-2} + e_{2n-3}, \quad \ldots \quad  u_{n} = e_{2} + e_{1}.
    \end{align*}
    Observe that the span of the $2n+2$ non-zero vectors $\{ u_{i}\}_{i=0}^{n} \cup \{ v_{i}\}_{i=0}^{n}$ is $(2n+1)$-dimensional. Additionally, we have the relation $v_{0} = \sum_{i=0}^{n} u_{i} - \sum_{i=1}^{n} v_{i}$ which involves \textit{all} the $u_i$'s and $v_i$'s. Therefore, every subset of $\{ u_{i}\}_{i=0}^{n} \cup \{ v_{i}\}_{i=0}^{n}$ of size $2n+1$ is linearly independent which shows that $\sigma\dual$ is a cone over a simplicial polytope, or equivalently, $\sigma$ is a cone over a simple polytope. There are a few possibilities for a simplicial polytope of dimension $2n$ with $2n+2$ vertices and they are completely classified in \cite{Grunbaum:polytope-book}*{\S6.1}. In particular, the type of the polytope is determined by the signs appearing in the linear relation between the $2n+2$ vectors, and $\sigma\dual$ is combinatorially a cone over $T_{n}^{2n}$, following Grünbaum's notation. This is combinatorially equivalent to the cyclic polytope $C(2n+2, 2n)$. Then the face number for the simple polytope obtained by a hyperplane section of $\sigma$ is given by
    $$ f_{l} = { 2n+2 \choose l+2} -2 { n+1 \choose l+2}.\footnote{We reindexed the face numbers from \cite{Grunbaum:polytope-book}*{6.1.2} so that the $f_{l}$'s match up with Theorem \ref{theo:sing-coho-simple}} $$
    This, along with Theorem \ref{theo:sing-coho-simple}, allows us to describe the Hodge--Du Bois diamond of the binomial hypersurface $\{ x_{0}\cdots x_{n} = y_{0} \cdots y_{n} \} \subset \PP^{2n+1}$. In particular, if we consider the singular cubic fourfold given by $\{ xyz = uvw \} \subset \PP^{5}$, the face numbers are $(f_{0}, f_{1}, f_{2}, f_{3}) = (9, 18, 15, 6)$ and we get the following description for the Hodge--Du Bois diamond:
    \begin{center}
    \begin{tabular}{c | c c c c c}
    4 & 0 & 0 & 0 & 0 & 1 \\
    3 & 0 & 1 & 4 & 5 & 0 \\
    2 & 0 & 0 & 1 & 0 & 0 \\
    1 & 0 & 1 & 0 & 0 & 0 \\
    0 & 1 & 0 & 0 & 0 & 0 \\
    \hline 
    $q/p$ & 0 & 1 & 2 & 3 & 4.
    \end{tabular}
    \end{center}
    The horizontal index represents $p$ and the vertical index represents $q$, in $\dim H^{q}(S, \duBois_{S}^{p})$. The singular cohomologies $H^{\ast}(S, \QQ)$ are all mixed of Hodge--Tate type, so this table determines the weight pieces on the singular cohomology as well.
    \end{exam}
    
    {\bf Acknowledgments.} We would like to thank Mircea Musta\c{t}\u{a} for numerous helpful discussions and Lei Xue for several discussions on polytopes. The second author thanks Mihnea Popa for asking questions about the depth of reflexive differentials that led us to some of the results in this paper.

    \appendix
    \section{Explicit calculation for $\QQ_{X}^{H'}$ in low dimensions} \label{sec:appendix-explicit-calculation}
	In the appendix, we give an explicit calculation for the coefficients showing up in Theorem \ref{theo:MHM-structure}. We calculate these numbers up to dimension 5, and address the limitation of our method in dimension 6. Of course, it is enough to deal with the affine case.
	
	\subsection{Dimension 0, 1, and 2} Since all toric varieties of dimension $\leq 2$ are simplicial, we have $\QQ_{X}^{H'} \simeq \IC_{X}^{H}$.
	
	\subsection{Dimension 3} Let $X$ be a 3-dimensional affine toric variety corresponding to a 3-dimensional cone $\sigma$. We know that $\lcdef(X) = 0$, so $\cH^{0}\QQ_{X}^{H'} = \QQ_{X}^{H'}$. Writing down Theorem \ref{theo:MHM-structure}, we have
	$$ \gr_{3}^{W} \QQ_{X}^{H'} \simeq \IC_{X}^{H}, \qquad \gr_{2}^{W} \QQ_{X}^{H'} \simeq \QQ_{x_{\sigma}}^{H}(-1)^{a_{\sigma}^{0,1}}.$$
	It is enough to determine the number $a_{\sigma}^{0,1}$. In this case, one can directly appeal to Theorem \ref{theo:cone-over-simplicial-polytope} and we get $a_{\sigma}^{0,1} = f_{1}(\sigma) -3$, where $f_{1}(\sigma)$ is the number of rays in $\sigma$.
	
	\subsection{Dimension 4} Let $X$ be a 4-dimensional affine toric variety corresponding to a cone $\sigma$. Let $\cP$ be the fan associated to $\sigma$. We list the terms in Theorem \ref{theo:MHM-structure}.
	\begin{enumerate}
		\item $\gr_{4}^{W} \cH^{0}\QQ_{X}^{H'} \simeq \IC_{X}^{H}$,
		\item $\gr_{3}^{W} \cH^{0} \QQ_{X}^{H'} \simeq \bigoplus_{\lambda \in \cP_{3}} \IC_{S_{\lambda}}^{H}(-1)^{a_{\lambda}^{0,1}}$,
		\item $\gr_{2}^{W} \cH^{0} \QQ_{X}^{H'} \simeq \QQ_{x_{\sigma}}^{H}(-1)^{a_{\sigma}^{0,1}}$,
		\item $\gr_{2}^{W} \cH^{-1} \QQ_{X}^{H'} \simeq \QQ_{x_{\sigma}}^{H}(-1)^{a_{\sigma}^{1,1}}$.
	\end{enumerate}
	The computation for $a_{\lambda}^{0,1}$ is already done in the previous section, and therefore it remains to determine the two numbers $a_{\sigma}^{0,1}$ and $a_{\sigma}^{1,1}$. Dualizing $\QQ_{X}^{H'}$ gives the following:
	\begin{enumerate}
		\item $\gr_{-4}^{W} \cH^{0} \DD\QQ_{X}^{H'} \simeq \IC_{X}^{H} (4)$,
		\item $\gr_{-3}^{W} \cH^{0} \DD \QQ_{X}^{H'} \simeq \bigoplus_{\lambda \in \cP_{3}} \IC_{S_{\lambda}}^{H}(2)^{a_{\lambda}^{0,1}}$,
		\item $\gr_{-2}^{W} \cH^{0} \DD\QQ_{X}^{H'} \simeq \QQ_{x_{\sigma}}^{H}(1)^{a_{\sigma}^{0,1}}$,
		\item $\gr_{-2}^{W} \cH^{1} \DD \QQ_{X}^{H'} \simeq \QQ_{x_{\sigma}}^{H}(1)^{a_{\sigma}^{1,1}}$.
	\end{enumerate}
	We use the spectral sequence
	$$ E_{2}^{p,q}(k) = \cH^{p}(\gr_{k} \DR \cH^{q} \DD \QQ_{X}^{H'}) \implies \SheafExt_{\cO_{X}}^{k+p+q}(\duBois_{X}^{k}, \omega_{X}).$$
	By examining the possible non-zero entries of $E_{2}^{p,q}$ this spectral degenerates at page $E_{2}$. We also note that the whole spectral sequence carries a natural grading by $M$, and we focus on the degree 0 part. By taking $k = 1$, we get
	$$a_{\sigma}^{1,1} = \dim E_{2}^{0, 1}(1) =  \dim \SheafExt_{\cO_{X}}^{2} (\duBois_{X}^{1}, \omega_{X})_{0} = \dim H^{2} \Ish_{\sigma}^{3}.$$
	From the spectral sequence
	$$ E_{1}^{r,s} = \left(\cH^{r+s} \gr_{1} \DR \gr_{-r}^{W} \cH^{0} \DD \QQ_{X}^{H'} \right)_{0} \implies \left( \cH^{r+s} \gr_{1} \DR \cH^{0} \DD \QQ_{X}^{H'}\right)_{0}, $$
    we can see that $a_{\sigma}^{0, 1} = \dim (\cH^{0}\gr_{1}\DR \cH^{0} \DD \QQ_{X}^{H'})_{0}$. Also, from the spectral sequence $E_{2}^{p, q}(1)$, we get
    $$ a_{\sigma}^{0, 1} =  \dim \left( \cH^{0} \gr_{1} \DR \cH^{0} \DD \QQ_{X}^{H'} \right)_{0} = \dim \SheafExt_{\cO_{X}}^{1}(\duBois_{X}^{1}, \omega_{X})_{0} = \dim H^{1}(\Ish_{\sigma}^{3}). $$
    This concludes the calculation of the 4-dimensional case.
	
	\subsection{Dimension 5.} We continue the computation in dimension 5. Let $X$ be a 5-dimensional toric variety corresponding to a cone $\sigma$ and let $\cP$ be the fan associated to $\sigma$. We list the terms as in the previous example. The one that we need is the behavior of the dual $\DD \QQ_{X}^{H'}$, so we list them.

    \begin{enumerate}
		\item $\gr_{-5}^{W} \cH^{0} \DD\QQ_{X}^{H'} \simeq \IC_{X}^{H}(5)$,
		\item $\gr_{-4}^{W} \cH^{0} \DD\QQ_{X}^{H'} \simeq \bigoplus_{\mu \in \cP_{3}} \IC_{S_{\mu}}^{H}(3)^{a_{\mu}^{0,1}} \oplus \QQ_{x_{\sigma}}^{H}(2)^{a_{\sigma}^{0,2}}$,
		\item $\gr_{-3}^{W} \cH^{0} \DD\QQ_{X}^{H'} \simeq \bigoplus_{\lambda \in \cP_{4}} \IC_{S_{\lambda}}^{H}(2)^{a_{\lambda}^{0,1}}$,
		\item $\gr_{-2}^{W} \cH^{0} \DD\QQ_{X}^{H'} \simeq \QQ_{x_{\sigma}}^{H} (1)^{a_{\sigma}^{0,1}}$,
		\item $\gr_{-3}^{W} \cH^{1} \DD\QQ_{X}^{H'} \simeq \bigoplus_{\lambda \in \cP_{4}} \IC_{S_{\lambda}}^{H}(2)^{a_{\lambda}^{1,1}}$,
		\item $\gr_{-2}^{W} \cH^{1} \DD\QQ_{X}^{H'} \simeq \QQ_{x_{\sigma}}^{H}(1)^{a_{\sigma}^{1,1}}$,
		\item $\gr_{-2}^{W} \cH^{2} \DD\QQ_{X}^{H'} \simeq \QQ_{x_{\sigma}}^{H}(1)^{a_{\sigma}^{2,1}}$.
	\end{enumerate}
    
	It is enough to determine $a_{\sigma}^{0,1}, a_{\sigma}^{1,1}, a_{\sigma}^{2,1}$, and $a_{\sigma}^{0,2}$.
	We record that
	$$ \dR_{\sigma}(K, L) = L^{-5} + \alpha L^{-3}K^{-1} + \beta L^{-1} K^{-2},$$
	where $\alpha$ and $\beta$ are some numbers that we can explicitly compute, which is combinatorially determined by $\sigma$ (see \S\ref{sec:ICTV-summary} for the definition of $\dR$). We start from the spectral sequence
	$$ E_{2}^{p,q}(k) = \cH^{p} \gr_{k} \DR \cH^{q} \DD \QQ_{X}^{H'} \implies \SheafExt_{\cO_{X}}^{k+ p+ q}(\duBois_{X}^{k}, \omega_{X}).$$
	The following are the possible non-zero slots of this spectral sequence
	$$ \begin{tikzcd}
		& & & & & \bullet \\
		& & & & \bullet & \bullet \\
		\bullet &\bullet & \bullet & \bullet & \bullet & \bullet,
	\end{tikzcd} $$
	where the right bottom slot is $(p, q) = (0,0)$. We see that the spectral sequence is degenerate at page $E_{2}$. First, by taking $k = 1$, it is immediate to see that
	\begin{align*}
        a_{\sigma}^{2, 1} & = \dim \SheafExt_{\cO_{X}}^{3}(\duBois_{X}^{1}, \omega_{X})_{0} = \dim H^{3} \Ish_{\sigma}^{4}, \\
		a_{\sigma}^{1,1} & = \dim \SheafExt_{\cO_{X}}^{2} (\duBois_{X}^{1}, \omega_{X})_{0} = \dim H^{2} \Ish_{\sigma}^{4} \\
		a_{\sigma}^{0, 1}& = \dim \SheafExt_{\cO_{X}}^{1}(\duBois_{X}^{1}, \omega_{X})_{0} = \dim H^{1} \Ish_{\sigma}^{4}.
	\end{align*}
    Indeed, $\IC_{X}^{H}(5), \IC_{S_{\mu}}^{H}(3),$ and $\IC_{S_{\lambda}}^{H}(2)$ have no degree zero part after taking the cohomologies of $\gr_{1}\DR$.
    
    It remains to compute $a_{\sigma}^{0, 2}$. For this, we put $k = 2$ and we see that
	$$ \dim \left(\cH^{0} \gr_{2} \dR \cH^{0} \DD \QQ_{X}^{H'}\right)_{0} + \dim \left(\cH^{-1} \gr_{2} \DR \cH^{1} \DD \QQ_{X}^{H'}\right)_{0} = \dim \SheafExt^{2}(\duBois_{X}^{2}, \omega_{X})_{0}.$$

    From the spectral sequence
    $$ E_{1}^{r, s} = \cH^{r+s} \gr_{2} \DR \gr_{-r}^{W} \cH^{1} \DD \QQ_{X}^{H'} \implies \cH^{r+s} \gr_{2} \DR \cH^{1} \DD \QQ_{X}^{H'},$$
    we see that
    $$ \dim \left(\cH^{-1} \gr_{2} \DR \cH^{1} \DD \QQ_{X}^{H'} \right)_{0} = \sum_{\lambda \in \cP_{4}} a_{\lambda}^{1,1} = \sum_{\lambda \in \cP_{4}} \dim H^{2}\Ish_{\lambda}^{3}.$$
    Also, we see that the spectral sequence computing the degree 0 part of $\cH^{j} \gr_{2} \DR \cH^{0} \DD \QQ_{X}^{H'}$ degenerates at page 2 and it is the cohomology of the following two term complex sitting in degrees $-1$ and $0$:
    $$ \left(\cH^{-1} \gr_{2} \DR \gr_{-3}^{W}\cH^{0} \DD \QQ_{X}^{H'} \right)_{0} \to \left(\cH^{0} \gr_{2} \DR \gr_{-4}^{W} \cH^{0} \DD \QQ_{X}^{H'}\right)_{0}. $$
    The dimension of the vector spaces are $\sum_{\lambda} a_{\lambda}^{0, 1}$ and $a_{\sigma}^{0, 2}$, respectively. Let $r$ be the rank of this map. Then we have
    \begin{align*}
        \sum_{\lambda \in \cP_{4}} a_{\lambda}^{0, 1} - r & = \dim H^{1}(\Ish_\sigma^{3}) ,\\
        a_{\sigma}^{0,2} - r + \sum_{\lambda \in \cP_{4}} a_{\lambda}^{1,1} & = \dim H^{2} (\Ish_{\sigma}^{3}).
    \end{align*}
	Therefore, we are able to determine $r$ and $a_{\sigma}^{0,2}$.
	
	\begin{rema}
		While following the computation, we have two inequalities that seem to be non-trivial. We know that
		$$ \sum_{\lambda \in \cP_{4}} a_{\lambda}^{0,1}= \sum_{\lambda \in \cP_{4}} \dim H^{1} \Ish_{\lambda}^{3} \geq \dim H^{1} \Ish_{\sigma}^{3} $$
		and
		$$ \sum_{\lambda \in \cP_{4}} a_{\lambda}^{1,1} = \sum_{\lambda \in \cP_{4}} \dim H^{2} \Ish_{\lambda}^{3} \leq \dim H^{2} \Ish_{\sigma}^{3}. $$
		It would be interesting if one could see this inequality directly by a combinatorial argument.
	\end{rema}

	\subsection{Dimension 6.} We give an attempt to compute the coefficients in Theorem \ref{theo:MHM-structure} in dimension 6. Let $X$ be a 6-dimensional toric variety corresponding to a cone $\sigma$ and let $\cP$ be the fan associated to $\sigma$. We list the terms for $\DD \QQ_{X}^{H'}$ as in the previous example.
	\begin{enumerate}
		\item $\gr_{-6}^{W} \cH^{0} \DD\QQ_{X}^{H'} \simeq \IC_{X}^{H}(6)$,
		\item $\gr_{-5}^{W} \cH^{0} \DD\QQ_{X}^{H'} \simeq \bigoplus_{\mu \in \cP_{3}} \IC_{S_{\mu}}^{H}(4)^{a_{\mu}^{0,1}} \oplus \bigoplus_{\tau \in \cP_{5}} \IC_{S_{\tau}}^{H}(3)^{a_{\tau}^{0,2}}$,
		\item $\gr_{-4}^{W} \cH^{0} \DD\QQ_{X}^{H'} \simeq\bigoplus_{\lambda \in \cP_{4}} \IC_{S_{\lambda}}^{H}(3)^{a_{\lambda}^{0, 1}} \oplus \QQ_{x_{\sigma}}^{H}(2)^{a_{\sigma}^{0, 2}}$,
		\item $\gr_{-3}^{W} \cH^{0} \DD\QQ_{X}^{H'} \simeq \bigoplus_{\tau \in \cP_{5}} \IC_{S_{\tau}}^{H}(2)^{a_{\tau}^{0, 1}}$,
		\item $\gr_{-2}^{W} \cH^{0} \DD\QQ_{X}^{H'} \simeq \QQ_{x_{\sigma}}^{H}(1)^{a_{\sigma}^{0,1}}$,
		\item $\gr_{-4}^{W} \cH^{1}\DD\QQ_{X}^{H'} \simeq \bigoplus_{\lambda \in \cP_{4}} \IC_{S_{\lambda}}^{H}(3)^{a_{\lambda}^{1,1}} \oplus \QQ_{x_{\sigma}}^{H}(2)^{a_{\sigma}^{1,2}}$,
		\item $\gr_{-3}^{W} \cH^{1} \DD\QQ_{X}^{H'} \simeq \bigoplus_{\tau \in \cP_{5}} \IC_{S_{\tau}}^{H}(2)^{a_{\tau}^{1,1}}$,
		\item $\gr_{-2}^{W} \cH^{1} \DD\QQ_{X}^{H'} \simeq \QQ_{x_{\sigma}}^{H}(1)^{a_{\sigma}^{1,1}}$,
		\item $\gr_{-3}^{W} \cH^{2} \DD\QQ_{X}^{H'} \simeq \bigoplus_{\tau\in \cP_{5}} \IC_{S_{\tau}}^{H}(2)^{a_{\tau}^{2, 1}}$,
		\item $\gr_{-2}^{W} \cH^{2} \DD\QQ_{X}^{H'} \simeq \QQ_{x_{\sigma}}^{H}(1)^{a_{\sigma}^{2, 1}}$,
		\item $\gr_{-2}^{W} \cH^{3} \DD\QQ_{X}^{H'} \simeq \QQ_{x_{\sigma}}^{H}(1)^{a_{\sigma}^{3, 1}}$.
	\end{enumerate}
	It is enough to determine $a_{\sigma}^{0, 1}, a_{\sigma}^{0, 2}, a_{\sigma}^{1, 1}, a_{\sigma}^{1, 2}, a_{\sigma}^{2, 1},$ and $a_{\sigma}^{3, 1}$. We point out that we are able to determine $a_{\sigma}^{0, 1}, a_{\sigma}^{1, 1}, a_{\sigma}^{2, 1}$, and $a_{\sigma}^{3, 1}$, but not $a_{\sigma}^{0,2}$ and $a_{\sigma}^{1,2}$.

    We start from the spectral sequence
	$$ E_{2}^{p,q}(k) = \cH^{p} \gr_{k} \DR \cH^{q} \DD \QQ_{X}^{H'} \implies \SheafExt_{\cO_{X}}^{k+p+q} (\duBois_{X}^{k}, \omega_{X}).$$
	We draw the possible non-zero slots of this spectral sequence
	$$ 
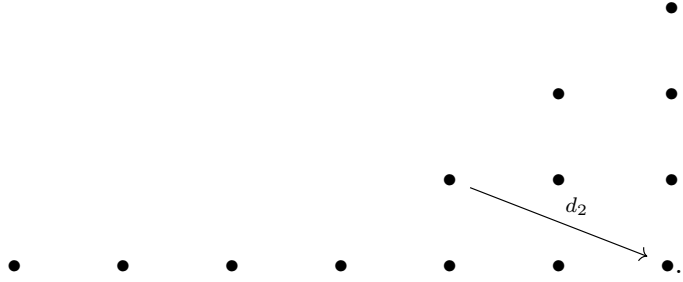
\begin{tikzcd}
		& & & & & & \bullet \\
		& & & & & \bullet & \bullet \\
		& & & &\bullet \ar[rrd, "d_{2}"]& \bullet & \bullet \\
		\bullet & \bullet &\bullet & \bullet & \bullet & \bullet & \bullet.
	\end{tikzcd} $$

    Using the same method as the 5-dimensional case, it is not so hard to see that
    \begin{align*}
        a_{\sigma}^{3,1}&= \dim H^{4} \Ish_{\sigma}^{5}, \\
        a_{\sigma}^{2, 1} &  = \dim H^{3} \Ish_{\sigma}^{5}, \\
        a_{\sigma}^{1, 1} & = \dim H^{2} \Ish_{\sigma}^{5}.
    \end{align*}
	
	We attempt to compute $a_{\sigma}^{0, 2}$ and $a_{\sigma}^{1, 2}$, hence we put $k = 2$ for $E_{2}^{p,q}(k)$. Considering
	$$ \widetilde{E}_{1}^{r, s} = \cH^{r +s} \left(\gr_{2} \DR  \gr_{-r}^{W} \cH^{2} \DD\QQ_{X}^{H'} \right) \implies \cH^{r + s} \left( \gr_{2} \DR \cH^{2} \DD\QQ_{X}^{H'}\right),$$
	we see that $\dim \cH^{-1}\left( \gr_{2} \DR \cH^{2} \DD \QQ_{X}^{H'} \right)_{0} = \sum_{\tau \in \cP_{5}} a_{\tau}^{2, 1}$.
	Considering
	$$ \widetilde{E}_{1}^{r, s} = \cH^{r +s} \left(\gr_{2} \DR  \gr_{-r}^{W} \cH^{1} \DD\QQ_{X}^{H'} \right) \implies \cH^{r + s} \left( \gr_{2} \DR \cH^{1} \DD\QQ_{X}^{H'}\right)$$
	on degree zero, we have a possible non-trivial map
	$$ d_{1}^{3, -4}: \cH^{-1} \left( \gr_{2} \DR \gr_{-3}^{W} \cH^{1} \DD \QQ_{X}^{H'} \right)_{0} \to \cH^{0}  \left(\gr_{2} \DR \gr_{-4}^{W} \cH^{1} \DD \QQ_{X}^{H'} \right)_{0}.$$
	Note that the dimension of the source is $\sum_{\tau \in \cP_{5}} a_{\tau}^{1,1}$ and the dimension of the target is $a_{\sigma}^{1, 2}$. Let the rank of this map be $r_{1}$. Then we have
	\begin{align*}
		& \dim \cH^{-2} \left(\gr_{2} \DR \cH^{1} \DD \QQ_{X}^{H'} \right)_{0} = 0 \\
		& \dim \cH^{-1} \left(\gr_{2} \DR \cH^{1} \DD \QQ_{X}^{H'} \right)_{0} = \sum_{\tau \in \cP_{5}} a_{\tau}^{1,1} - r_{1} \\
		& \dim \cH^{0} \left(\gr_{2} \DR \cH^{1} \DD \QQ_{X}^{H'} \right)_{0} = a_{\sigma}^{1, 2} - r_{1}.
	\end{align*}
	This shows that 
	$$d_{2}^{-2, 1} : \cH^{-2} \left(\gr_{2} \DR \cH^{1} \DD \QQ_{X}^{H'}\right)_{0} \to \cH^{0} \left( \gr_{2} \DR \cH^{0} \DD \QQ_{X}^{H'} \right)_{0}$$
	is the zero map.
	Finally, by considering the same spectral sequence for $\cH^{0} \DD \QQ_{X}^{H'}$, we get
	\begin{align*}
		& \dim \cH^{-1} \left(\gr_{2} \DR \cH^{0} \DD \QQ_{X}^{H'} \right)_{0} = \sum_{\tau \in \cP_{5}} a_{\tau}^{0,1} - r_{0} \\
		& \dim \cH^{0} \left(\gr_{2} \DR \cH^{0} \DD \QQ_{X}^{H'} \right)_{0} = a_{\sigma}^{0, 2} - r_{0},
	\end{align*}
    where $r_{0}$ is the rank of the map
    $$ d_{1}^{3, -4}: \cH^{-1} \left( \gr_{2} \DR \gr_{-3}^{W} \cH^{0} \DD \QQ_{X}^{H'} \right)_{0} \to \cH^{0}  \left(\gr_{2} \DR \gr_{-4}^{W} \cH^{0} \DD \QQ_{X}^{H'} \right)_{0}.$$
    
	Wrapping all up, we have
	\begin{align*}
		\dim H^{3} \Ish_{\sigma}^{3} & = \sum_{\tau \in \cP_{5}} a_{\tau}^{2, 1} + \left( a_{\sigma}^{1, 2} - r_{1} \right) \\
		\dim H^{2} \Ish_{\sigma}^{3} & = \left( \sum_{\tau \in \cP_{5}} a_{\tau}^{1,1} - r_{1} \right) + \left( a_{\sigma}^{0, 2} - r_{0} \right) \\
		\dim H^{1} \Ish_{\sigma}^{3} & = \sum_{\tau \in \cP_{5}} a_{\tau}^{0, 1} - r_{0}.
	\end{align*}
	We have three equations, with four indeterminacies $r_{0}, r_{1}, a_{\sigma}^{0, 2}$, and $a_{\sigma}^{1,2}$. This is the reason we are not able to determine the numbers $a_{\sigma}^{1, 2}$ and $a_{\sigma}^{0, 2}$.

	\bibliographystyle{alpha}
	\bibliography{Reference}

\end{document}